\documentclass[reqno, 11pt]{amsart}
\usepackage{mathtools}
\usepackage{amsmath}
\usepackage{amssymb}
\usepackage{yhmath}
\usepackage{graphicx}
\usepackage{ mathrsfs }
\usepackage{bbm}
\usepackage{xcolor}
\usepackage{tikz-cd}
\DeclareMathAlphabet{\mathpzc}{OT1}{pzc}{m}{it}

\newtheorem{theorem}{Theorem}[section]

\newtheorem*{claim*}{Claim}

\newtheorem{lemma}[theorem]{Lemma}
\newtheorem{lem}[theorem]{Lemma}

\newtheorem{cor}[theorem]{Corollary}

\newtheorem{prop}[theorem]{Proposition}
\newtheorem{Thm}[theorem]{Theorem}
\newtheorem{thm}[theorem]{Theorem}

\theoremstyle{definition}
\newtheorem{Def}[theorem]{Definition}

\theoremstyle{remark}

\newtheorem{Rmk}[theorem]{Remark}

\numberwithin{equation}{section}

\newcommand{\norm}[1]{\lVert#1\rVert}
\newcommand{\op}{\operatorname}
\newcommand{\Om}{\Omega}

\newcommand{\bb}{\mathbb}

\newcommand{\be}{\begin{equation}}
\newcommand{\ee}{\end{equation}}
\newcommand{\Ga}{\Gamma}

\newcommand{\ga}{\gamma}

\newcommand{\La}{\Lambda}
\newcommand{\inte}{\op{int}}
\newcommand{\ba}{\backslash}

\newcommand{\cal}{\mathcal}
\newcommand{\br}{\mathbb R}

\newcommand{\F}{\cal F}

\newcommand{\G}{\Gamma}
\newcommand{\m}{\mathsf{m}}

\newcommand{\T}{\op{T}}
\renewcommand{\frak}{\mathfrak}
\renewcommand{\u}{\mathsf u}
\renewcommand{\v}{\mathsf v}
\newcommand{\e}{\varepsilon}
\newcommand{\BR}{\op{BR}}

\newcommand{\BMS}{\op{BMS}}
\renewcommand{\L}{\mathcal L}
\newcommand{\fa}{\mathfrak a}
\newcommand{\pc}{P^{\circ}}
\renewcommand{\i}{\op{i}}
\newcommand{\M}{\mathsf M}

\begin{document}

\title[The Hopf-Tsuji-Sullivan dichotomy]{The Hopf-Tsuji-Sullivan dichotomy in higher rank and applications to Anosov subgroups}

\author{Marc Burger}
\address{Department of Mathematics, ETH, Zurich, Switzerland}
\email{burger@ethz.ch}
\author{Or Landesberg}
\address{Department of Mathematics, Yale University, New Haven, CT 06520}
\email{or.landesberg@yale.edu}

\author{Minju Lee}
\address{Mathematics department, Yale university, New Haven, CT 06520, and Department of Mathematics, University of Chicago, Chicago, IL 60637 (current address)}
\email{minju1@uchicago.edu}

\author{Hee Oh}
\address{Department of Mathematics, Yale University, New Haven, CT 06520}
\email{hee.oh@yale.edu}
\thanks{Oh is partially supported by the NSF}

\maketitle
\begin{abstract}
We establish an extension of the Hopf-Tsuji-Sullivan dichotomy to
any Zariski dense discrete subgroup of a semisimple real algebraic group $G$.
We then apply this dichotomy to Anosov subgroups  of $G$, which surprisingly presents a different phenomenon depending on the rank of  the ambient group $G$.
\end{abstract} 
\tableofcontents
\section{Introduction}
Let $G$ be a connected simple real algebraic group of rank one,  $(X,d)$ the associated Riemannian symmetric space and
  $\partial X$ the geometric boundary of $X$.
  We fix a base point $o\in X$, and $\pi:\T^1(X)\to X$ denotes the canonical projection of a vector to its basepoint.
  The Hopf parametrization of the unit tangent bundle $\T^1(X)$ maps a vector $v\in \T^1(X)$
  to $$(v^+, v^-, \beta_{v^+}(o, \pi(v)))$$
  where $v^+,v^-\in \partial X$ are respectively the forward and backward endpoints of the geodesic determined by $v$ and  for $\xi\in \partial X$ and $x,y\in X$, $\beta_\xi (x,y)$ denotes the Busemann function given by $$\beta_\xi(x,y)=\lim_{z\to \xi} d(y, z)-d(x,z).$$
  This gives a homeomorphism
$$\T^1(X)\simeq (\partial X \times \partial X-\Delta (\partial X))\times \br$$  where
$\Delta (\partial X)$ denotes the diagonal embedding of $\partial X$ into $\partial X\times \partial X$ and the geodesic flow $\mathcal G^t$ on $\T^1(X)$ corresponds to the translation flow on $\br$.

 \subsection*{The Hopf-Tsuji-Sullivan dichotomy in rank one} Let $\G<G$ be a non-elementary discrete subgroup. A Borel probability measure $\nu$ on $\partial X$ is called a $\Ga$-conformal measure of dimension $\delta\ge 0$ if for any $\ga\in \Ga$ and $\xi\in \partial X$,
  $$\frac{d\gamma_*\nu}{d\nu}(\xi)= e^{-\delta \beta_\xi(o,\ga (o))}$$
where  $\ga_* \nu(Q)=\nu(\gamma^{-1} Q)$ for any Borel subset $Q\subset \partial X$.

  Each $\Gamma$-conformal  measure $\nu$ on $\partial X$ determines a unique
 geodesic flow invariant Borel measure $m_\nu$ on $\T^1(\Gamma\ba X)$, which is locally equivalent to $\nu\otimes \nu \otimes ds$ in the Hopf coordinates, where $ds$ denotes the Lebesgue measure on $\br$. 
The following criterion known as the {\it Hopf-Tsuji-Sullivan dichotomy} 
 relates  dynamical properties of the geodesic flow $\cal G^t$ with respect to the measure $m_\nu$, the $\nu$-size of the conical\footnote{A point $\xi\in \partial X$ is called a conical limit point of $\Ga$ if a geodesic ray toward $\xi$ accumulates in $\Gamma\ba X$.} limit points of $\G$ and the divergence property of the Poincare series 
 $\cal P(s)=\sum_{\ga\in \Ga} e^{-s d(\ga o, o)}$ at the dimension of $\nu$: we denote by $\La_\mathsf{con}\subset \partial X$ the set of all conical limit points of $\Ga$.
\begin{Thm}\label{main} Let $G$ be a connected simple real algebraic group of rank one and $\G<G$ a non-elementary discrete subgroup.
Let $\nu$ be a $\Ga$-conformal measure on $\partial X$ of dimension $\delta$. The following are equivalent: \begin{enumerate}
\item $\nu(\Lambda_\mathsf{con})>0$ (resp.  $\nu(\Lambda_{\mathsf{con}})=0$);
\item  $\nu(\Lambda_\mathsf{con})=1$ (resp.  $\nu(\Lambda_\mathsf{con})=0$); 
\item the geodesic flow $\cal G^t$ is conservative (resp. completely dissipative) with respect to $m_\nu$;
\item the geodesic flow $\cal G^t$ is ergodic (resp. non-ergodic) with respect to $m_\nu$;
    \item  $\sum _{\ga\in \Gamma} e^{-\delta  d(o, \gamma o)}=\infty $ (resp. 
    $\sum _{\ga\in \Gamma} e^{-\delta  d(o, \gamma o)}<\infty )$ where $\delta$ is the conformal dimension of $\nu$ and $o\in X$ is any point.
     \end{enumerate}
\end{Thm}
  
Most equivalences are due to Sullivan for real hyperbolic spaces  \cite{Su} (see also \cite{Ts}, \cite{AS}) and to Burger-Mozes for proper CAT (-1) spaces \cite[Sec. 6.3]{BM}
and its complete form can be found in Nicholl's book
\cite[Ch. 8]{Ni} when $X$ is a real hyperbolic space and in Roblin's thesis \cite[Thm. 1.7]{Rob} for a proper CAT (-1) spaces.

We denote by $\La\subset \partial X$ the limit set of $\Ga$, which is the unique $\Ga$-minimal subset of $\partial X$ and by $\delta_\Ga$ the critical exponent of $\Ga$, that is, the abscissa of the convergence of the Poincare series $\cal P(s)$ of $\Ga$. The group $\G$ is called a divergent type if $\cal P(\delta_\Ga)=\infty$.
Patterson and Sullivan constructed a $\Ga$-conformal measure,
say, $\nu_{\op{PS}}$, supported on the limit set $\La$ of dimension $\delta_\Ga$, called the Patterson-Sullivan measure. Theorem \ref{main} implies that whether $\G$ is of divergent type or not is completely determined by the positivity of $\nu_{\op{PS}}(\La_\mathsf{con})$, and vice versa.

\subsection*{The case of convex cocompact groups}
A discrete group with $\La=\La_\mathsf{con}$ is called a {\it convex cocompact} subgroup. They are also characterized by the property that $\G$ acts cocompactly on the convex hull of $\La$ in $X$.
For a convex cocompact subgroup $\Gamma$, there exists a unique $\Gamma$-conformal measure supported
on the limit set $\Lambda$, namely the Patterson-Sullivan measure $\nu_{\op{PS}}$. The associated geodesic flow invariant measure on $\T^1(\Ga\ba X)$, called the Bowen-Margulis-Sullivan measure, is known to be the measure of maximal entropy \cite{Su}. 
An immediate consequence of Theorem \ref{main} for convex cocompact groups is as follows:

\begin{thm} Let $\G<G$ be a convex cocompact subgroup. Then
\begin{enumerate}
    \item the geodesic flow $\cal G^t$ on $\T^1(\Ga\ba X)$
    is
   conservative and ergodic with respect to the Bowen-Margulis-Sullivan measure $m^{\BMS}$;
   \item $\Ga$ is of divergent type, i.e., $\sum _{\ga\in \Gamma} e^{-\delta_\Ga  d(o, \gamma o)}=\infty $.
\end{enumerate}

\end{thm}

 The unit tangent bundle of $\Gamma\ba X$ is a double quotient space
$\Gamma\ba G/M$ where $M$ is a compact subgroup of $G$ commuting with the one-parameter diagonal subgroup $\{a_t\}$ which induces the geodesic flow.
 When $\Gamma$ is Zariski dense in addition, the lifted Bowen-Margulis-Sullivan measure, considered as an $M$-invariant measure on $\Gamma\ba G$,
  is also ergodic for the diagonal flow $\{a_t\}$ whenever $M$ is connected \cite[Thm. 1.1]{Win}. The only case of $M$ disconnected is when $G\simeq \op{SL}_2(\br)$ and $M=\{\pm e\}$, in which case $m^{\BMS}$ has at most two ergodic components \cite{LO2}.

\subsection*{The Hopf-Tsuji-Sullivan dichotomy in higher rank} The main aim of this article is to extend the Hopf-Tsuji-Sullivan dichotomy for discrete subgroups of higher rank semisimple real algebraic groups $G$, while replacing the geodesic flow of the rank one space with any one-parameter subgroup of diagonal elements of $G$ (Theorem \ref{dio}). 
Each one-parameter subgroup of diagonal elements corresponds to a direction, say, $\mathsf u$, in the positive Weyl chamber of $G$. We introduce
  the {\it $\mathsf u$-directional} conical limit set and  {\it $\mathsf u$-directional} Poincare series, whose properties relative to a given $\G$-conformal density is shown to determine ergodic properties
  of the action of the one-parameter subgroup $\{\exp (t\mathsf u):t\in \br\}$ with respect to an associated measure on $\Gamma\ba G$. We then apply the dichotomy  together with recent local mixing results of Chow and Sarkar \cite{CS} to Anosov subgroups $\Gamma$. We discover a surprising phenomenon that the rank of the ambient group $G$ dictates a completely opposite behavior for $\Ga$ as stated in Theorem \ref{thm.Ano}. We also deduce recurrent properties of the Burger-Roblin measures for each interior direction of the limit cone of $\Gamma$ (Corollary \ref{bb2}), which plays an important role in the recent measure classification result of Landesberg, Lee, Lindenstrauss and Oh \cite[Thm. 1.1] {LLLO}.
 
In order to state these results precisely, we now let $G$ be a connected, semisimple real algebraic group. Let $P$ be a minimal parabolic subgroup of $G$
with a fixed Langlands decomposition $P=MAN$. Here $A$ is a maximal real split torus of $G$, $M$ is a compact subgroup commuting with $A$ and $N$ is a maximal horospherical subgroup. We fix a positive Weyl chamber $\frak a^+\subset\frak a=\op{Lie}(A)$
so that $\log N$ consists of positive root subspaces.
 We fix a maximal compact subgroup $K<G$ so that the Cartan decomposition $G=K(\exp \fa^+) K$ holds, and denote by $\mu:G\to \fa^+$ the Cartan projection, i.e., for $g\in G$, $\mu(g)\in \fa^+$ is the unique element such that $g\in K\exp \mu(g)K$.

Let $\G<G$ be a Zariski dense discrete subgroup of $G$. 
 We denote by $\L_\Ga\subset \fa^+$ the limit cone of $\Gamma$, which is the asymptotic cone of $\mu(\Ga)$. Benoist showed that
 $\L_\Ga$ is a convex cone with non-empty interior \cite{Ben}.
Let $\F$ denote the Furstenberg boundary $G/P$ and $\La\subset \F$ the limit set of $\Ga$, which is the unique $\Gamma$-minimal subset.
For a linear form $\psi\in \fa^*$,  a Borel probability measure $\nu_\psi$ on $\cal F$ is called a $(\Ga,\psi)$-conformal measure
 if for any $\ga\in \G$ and $\xi\in \cal F$,
 \begin{equation}\label{eq.conf}
 \frac{d\ga_* \nu_\psi}{d\nu_\psi}(\xi) =e^{\psi (\beta_\xi (e, \gamma))}
 \end{equation}
  where $\beta$ denotes the $\frak a$-valued Busemann function (see Def. \ref{Bu}). Quint showed in \cite[Thm. 8.1]{Quint2} that a $(\Ga,\psi)$-conformal measure may exist only when $\psi\ge \psi_\Ga$ where  $\psi_\Ga:\fa\to \br$ denotes the growth indicator function of $\Ga$ (Def. \ref{def.GI}). Moreover, he constructed a $(\Gamma,\psi)$-conformal measure supported on $\La$ for
every linear form $\psi\ge \psi_\Ga$ satisfying
$\psi({\mathsf v})=\psi_\Ga({\mathsf v})$ for some $\mathsf{v}\in \L_\Ga\cap\inte\fa^+$ \cite[Thm. 8.4]{Quint2}. 

Let $\i:\fa^+\to \fa^+$ denote the opposition involution given by $\i (\mathsf v)= -\op{Ad}_{w_0} (\mathsf v)$ where $w_0$ is the longest Weyl element. In rank one groups, $\i$ is the identity map.
Letting $\F^{(2)}$ denote the unique open diagonal $G$-orbit in $\cal F\times \cal F$, the quotient space $G/M$ is homeomorphic to
$\F^{(2)}\times \fa$ via the Hopf parametrization which maps
$gM$ to $(gP, gw_0P, \beta_{gP}(e, g))$ for any $g\in G$.

For a given pair of $\G$-conformal measures $\nu_{\psi}$ and $\nu_{\psi\circ \i}$ on $\F$ with respect to $\psi$ and
$\psi\circ \i$ respectively, one can use the Hopf parameterization
to define a non-zero $A$-invariant Borel measure $\m(\nu_{\psi}, \nu_{\psi\circ {\i}})$ on the quotient space $\G\ba G/M$, which is locally equivalent to $d\nu_{\psi}\otimes d\nu_{\psi\circ {\i}}\otimes db$ in the Hopf coordinates, where $db$ denotes the Lebesgue measure on $\fa$;
we will call it the Bowen-Margulis-Sullivan measure (or simply $\BMS$-measure) associated to the pair $(\nu_\psi, \nu_{\psi\circ \i})$ (Section \ref{sec.TG}).  For simplicity, we write
$\m_\psi$ for $ \m(\nu_{\psi}, \nu_{\psi\circ {\i}})$, although the measure depends on the choice of conformal measures $\nu_\psi$ and $\nu_{\psi\circ {\i}}$, not only on $\psi$.

For $\mathsf{u}\in \inte \fa^+$, we will say that $\m_\psi$ is {\it $\mathsf{u}$-balanced} if
\be\label{bald} \limsup_{T\to \infty} \frac{\int_0^T \mathsf m_\psi (\cal O_1\cap \cal O_1\exp (t{\mathsf u})) \,dt }{\int_0^T \mathsf m_\psi (\cal O_2\cap \cal O_2\exp (t{\mathsf u})) \,dt }<\infty\ee  for any bounded Borel subsets $\cal O_i\subset \Ga\ba G/M$ with $\Omega\cap \inte \cal O_i\ne \emptyset$,
where $\Omega=\{[g]\in \Ga\ba G/M: gP, gw_0P\in \Lambda\}$.

Each BMS measure $\m_\psi$ on $\Ga\ba G/M$ can be considered as an $AM$-invariant measure on $\Ga\ba G$,  which we will also denote by $\m_\psi$, by abuse of notation.
While the set $\cal E=\{[g]\in \Gamma\ba G: gP\in \La\}$ is the unique $P$-minimal subset of $\Ga\ba G$, it breaks into finitely many $\pc$-minimal subsets in general where $\pc$ denotes the identity component of $P$. For each $\pc$-minimal subset $Y\subset \G\ba G$, the restriction $\m_\psi|_Y$ gives an $A$-invariant measure.

The conical limit set $\La_\mathsf{con}$ of $\Ga$
is given by
\be\label{con} \La_\mathsf{con}:=\{gP\in\F : \limsup
\Gamma g A^+  \ne \emptyset\}\ee
where $A^+=\exp \fa^+$ and $\limsup$ denotes the topological limit superior, i.e.~all accumulation points of the given family of sets.

\begin{Def} [Directional  conical limit set] \label{conl}
 For each $\mathsf{u}\in \inte \fa^+$,
we define the set of $\mathsf{u}$-directional conical limit points as follows:
$$
\La_{\mathsf u}:=\{gP\in\F : \limsup_{t\to +\infty}
\Gamma g \exp(t{\mathsf u})\ne \emptyset\};
$$
 this is a dense Borel measurable subset of $\La_\mathsf{con}$ if non-empty.
 \end{Def}
  It is easy to see that $\La_{\mathsf u}\ne \emptyset$
only when $\mathsf{u}\in \L_\Ga$.

For $R>0$ and $\mathsf{u}\in \inte \fa^+$, we define the following tube-like subset of $\Ga$ whose Cartan projection lies within distance $R$ from the ray $\br_+{\mathsf u}$:
\begin{equation*}\label{eq.D1}
\Ga_{{\mathsf u},R}:=\{\ga\in\Ga : \norm{\mu(\ga)-t{\mathsf u}}<R\quad \text{ for some }t\geq 0\},\end{equation*}
where $\|\cdot\|$ is an Euclidean norm on $\fa$.
%\begin{Def}  We say that $\Gamma$ is regular if {for any} simple root $\alpha$ of $\fa$,
%$\alpha(\mu(\ga))\to\infty$ whenever $\ga\to\infty$ in $\Ga$. \end{Def} 
%Note that if $\L_\Ga-\{0\}\subset \inte \fa^+$, then $\Ga$ is regular. In particular, in rank one Lie groups, all subgroups are regular. In the higher rank case, the regularity assumption on $\Ga$ is
%restrictive (for examples, lattices are never regular by the Poincare recurrence theorem) but important for us in several aspects. For regular subgroups, 
%each $\Gamma$-orbit in the symmetric space $X=G/K$ accumulates on the Furstenberg boundary $\cal F$ in the sense of
%Definition \ref{acc} and
%we have the shadow lemma (\cite[Lem. 7.8]{LO}, see Lemma \ref {lem.shadow}),
%which is a basic tool in the proof of the following theorem, which extends Theorem \ref{main} to all higher rank groups:
The following theorem extends Theorem \ref{main} to all Zariski dense subgroups of higher rank semisimple real algebraic groups:
\begin{theorem}[The Hopf-Tsuji-Sullivan dichotomy in higher rank]\label{thm.Rob}\label{dio}
Let $G$ be a connected semisimple real algebraic group and $\Ga<G$ be a Zariski dense discrete subgroup. Fix $\psi\in \fa^*$ and
let $\nu_\psi, \nu_{\psi \circ \i}$ be a pair of $(\Gamma,\psi)$ and $(\Ga, \psi\circ \i)$-conformal measures respectively, and let $\m_\psi=\m(\nu_{\psi}, \nu_{\psi\circ \i})$ denote the associated $\op{BMS}$ measure
on $\Ga\ba G/M$.
For any $\mathsf{u}\in \inte \fa^+$,
the following conditions (1)-(5) are equivalent and imply (6).
If  $\psi({\mathsf u})>0$ and $\m_\psi$ is $\mathsf{u}$-balanced, then 
 (6) implies (7). Moreover, the first cases of (1)-(7) can occur only when $\psi({\mathsf u})=\psi_\Ga({\mathsf u})$.
\begin{enumerate}
\item  $\max (\nu_{\psi}(\La_{\mathsf u}),\nu_{\psi\circ \i} (\La_{\i({\mathsf u})})) >0$ (resp. $\nu_{\psi}(\La_{\mathsf u})=0=\nu_{\psi\circ \i} (\La_{\i({\mathsf u})})$);
\item  $\max (\nu_{\psi}(\La_{\mathsf u}),\nu_{\psi\circ \i} (\La_{\i({\mathsf u})})) =1$ (resp. $\nu_{\psi}(\La_{\mathsf u})=0=\nu_{\psi\circ \i} (\La_{\i({\mathsf u})})$);
\item $(\Gamma\ba G/M, \{\exp (t{\mathsf u})\},  \mathsf m_\psi)$ is conservative (resp. totally dissipative);
\item $(\Gamma\ba G/M, \{\exp (t{\mathsf u})\},  \mathsf m_\psi)$ is ergodic (resp. non-ergodic); 
\item  For some (and hence for all) $\pc$-minimal subset $Y\subset \Ga\ba G$, the system
$(Y, \{\exp (t{\mathsf u})\},  \mathsf m_\psi|_{Y})$  is ergodic and conservative (resp. $\mathsf m_\psi (Y)=0$, or non-ergodic and totally dissipative);
\item 
$\sum_{\ga\in\Ga_{{\mathsf u},R}}e^{-\psi(\mu(\ga))}=\infty$ for some $R>0$ $($resp. $\sum_{\ga\in\Ga_{{\mathsf u},R}}e^{-\psi(\mu(\ga))}<\infty$ for all $R>0)$;
\item $\nu_{\psi}(\La_{\mathsf u})=1 =\nu_{\psi\circ \i} (\La_{\i({\mathsf u})})$
(resp. $\nu_{\psi}(\La_{\mathsf u})=0=\nu_{\psi\circ \i} (\La_{\i({\mathsf u})})$).
\end{enumerate}
\end{theorem}

\begin{Rmk}
\begin{enumerate}
\item When $G$ has rank one,  $\psi\circ \i =\psi$ for any $\psi\in \fa^*$, as the opposition involution $\i$ is trivial. Moreover, the
$\m_\psi$ being $\u$-balanced condition is not needed for the implication $(6)\Rightarrow (7)$.
For $\Gamma$ non-elementary, (1)-(7) are all equivalent to each other, except for (5), and for $\Gamma$ Zariski dense, these conditions imply (5).
\item When the rank of $G$ is at least $2$, we need $\Gamma$ to be Zariski dense for the equivalence of (3) and (4). The reason is that, when $\Gamma$ is not Zariski dense, the Jordan projection of $\Gamma$ may not generate a dense subgroup of $A$ while in the rank one case, the Jordan projection of any non-elementary subgroup generates a dense subgroup of $A$ {\cite{Ki}}.
%\item We mention that the equivalences among (3), (4), and (5) do not require the regularity assumption on $\Ga$.
\item We emphasize here that although the implication
$(3) \Rightarrow (1)$ is a direct consequence of the definition of $\m_\psi$, the proof for $(3) \Rightarrow (7)$ under the further $\u$-balanced condition of $\m_\psi$ requires the discussion of the directional Poincare series.
\end{enumerate}\end{Rmk}

 For discrete subgroups of a product of two rank one Lie groups whose projection to each factor is convex cocompact, Burger announced that
 $\nu_\psi(\La_{\mathsf u})=1$ for all $\psi\in\fa^*$ and $\mathsf{u}\in \inte \L_\Ga$ such that $\psi({\mathsf u})=\psi_\Ga({\mathsf u})$ \cite[Thm. 3]{Bu}.  Indeed, we show that this is a special case of a more general phenomenon which holds for all Anosov subgroups whose ambient group has rank at most $3$.

\subsection*{The case of Anosov subgroups} 
Although there are notions of Anosov subgroups with respect to a general
parabolic subgroup \cite{GW}, we will restrict our attention only to those Anosov subgroups with respect to a minimal parabolic subgroup. Recall that
a Zariski dense discrete subgroup $\Ga<G$ is an \textit{Anosov subgroup} (with respect to a minimal parabolic subgroup $P$) if it is a  finitely generated word hyperbolic group which admits a $\Ga$-equivariant embedding $\zeta$ of the Gromov boundary $\partial \Ga$  into $\cal F$ such that $(\zeta(x),\zeta(y))\in\cal F^{(2)}$ for all $x\ne y$ in $\partial\Ga$ \cite[Prop. 2.7 and Thm. 1.5]{GW}.
We note that Zariski dense images of representations of a surface subgroup in the Hitchin component \cite{La} as well as Schottky subgroups provide ample examples of Anosov subgroups (\cite[Prop. 3.3]{Q4}, see also \cite[Lem. 7.2]{ELO}). 
Let $\G$ be an Anosov subgroup for the rest of the introduction.
Set $$ D_\Ga^\star:= \{\psi\in \fa^*: \psi\ge \psi_\Ga, \psi({\mathsf v})=\psi_\Gamma({\mathsf v})\text{ for some  $\mathsf{v}\in\L_\Ga\cap \inte\fa^+$}\} .$$
 For each $\psi\in D_\Ga^\star$, there exists a unique unit vector $\mathsf{v}\in \L_\Ga\cap \inte \fa^+$ such that $\psi({\mathsf v})=\psi_\Gamma({\mathsf v})$ and $\mathsf{v}$ necessarily belongs to $\inte\L_\Ga$ (\cite[Prop. 4.11]{PS} and \cite[Lem. 4.3(i)]{Q4}, see also \cite[Lem. 4.3]{Samb3} and \cite[Thm. A.2(3)]{Car}).

Moreover, for each $\psi\in D_\Ga^\star$,
there exists a unique $(\Ga, \psi)$-conformal probability measure, say $\nu_\psi$, {\it supported on $\La$}
and the map $\psi\mapsto \nu_\psi$
is a homeomorphism between $D_\Ga^\star$ and
the space $\mathcal S_\Ga$ of all $\Gamma$-conformal probability measures supported on $\La$; hence $\cal S_\Ga$ is homeomorphic to the set of unit vectors of $\inte\L_\Ga$ (see \cite[Thm. 1.3]{LO} and references therein).
It was also shown in (\cite{LO}, \cite{LO2}) that for any $\psi\in D_\Ga^\star$ and $\m_\psi =\m(\nu_\psi, \nu_{\psi, \circ \i})$,
\begin{itemize}
    \item $\La=\La_{\mathsf{con}}$;
    \item for any $\pc$-minimal subset $Y\subset \Ga\ba G$, $\m_\psi |_Y$ is $A$-ergodic;
\item $\sum_{\ga\in\Ga} e^{-\psi(\mu(\ga))}=\infty$.
\end{itemize}
 
 On the other hand, the divergence of the directional Poincare 
  series (i.e., $\sum_{\ga\in\Ga_{{\mathsf u},R}}e^{-\psi(\mu(\ga))}$ for some $R>0$) turns out to depend on the rank of $G$:
\begin{theorem} \label{thm.Ano}
Let $\Ga<G$ be an Anosov subgroup. For any $\psi\in D_\Ga^\star$ and $\mathsf{u}\in \inte\fa^+$, the following conditions are equivalent and the first cases of (1)-(4) can occur only when $\mathsf{u}\in \inte\L_\Ga$:
\begin{enumerate}
\item $\op{rank } G\le 3$ and $\psi({\mathsf u})=\psi_\Gamma ({\mathsf u})$ (resp. $\op{rank } G>3$ or $\psi({\mathsf u})\ne \psi_\Gamma ({\mathsf u})$);
\item  $\nu_{\psi}(\La_{\mathsf u})=1=\nu_{\psi\circ \i} (\La_{\i({\mathsf u})}) $ (resp. $\nu_{\psi}(\La_{\mathsf u})=0=\nu_{\psi\circ \i} (\La_{\i({\mathsf u})})$);
\item  For some (and hence for all) $\pc$-minimal subset $Y\subset \Ga\ba G$, the system
$(Y, \{\exp (t{\mathsf u})\},  \mathsf m_\psi|_{Y})$  is ergodic and conservative (resp. non-ergodic and totally dissipative);
\item 
$\sum_{\ga\in\Ga_{{\mathsf u},R}}e^{-\psi(\mu(\ga))}=\infty$ for some $R>0$
 (resp. $\sum_{\ga\in\Ga_{{\mathsf u},R}}e^{-\psi(\mu(\ga))}<\infty$ for all $R>0$).
\end{enumerate}
\end{theorem}
For $\psi\in D_\Ga^\star$ and $\mathsf{u}\in \inte\L_\Ga$ with  $\psi({\mathsf u})=\psi_\Gamma ({\mathsf u})$, Chow and Sarkar proved in \cite{CS} the following local mixing result that for any
$f_1, f_2\in C_c(\Ga\ba G)$,
\be\label{cs} \lim_{t\to +\infty} t^{(\op{rank}G -1)/2}\int_{\Ga\ba G} f_1(x \exp t{\mathsf u}) f_2(x) d\m_\psi(x) =\kappa_{\mathsf u}\, \m_\psi (f_1)\m_\psi(f_2)\ee
for some constant $\kappa_{\mathsf u}>0$ depending only on $\mathsf{u}$ (see \cite{Samb} where this is proved for $M$-invariant functions for some special cases).

 Using the shadow lemma (Lemma \ref{lem.shadow}), we deduce from this local mixing result \eqref{cs} that
the $\mathsf{u}$-directional Poincare series 
$\sum_{\gamma\in \Ga_{{\mathsf u}, R}} e^{-\psi(\mu(\ga))}$ diverges if and only if $\op{rank} G\le 3$. Together with Theorem \ref{dio}, this implies Theorem \ref{thm.Ano}.

Let $\m_\psi^{\BR}$ denote the Burger-Roblin measure associated to $\nu_\psi$, that is,
the $MN$-invariant measure on $\Ga\ba G$ which is induced from a measure on $G/M$ locally equivalent to $d\nu_\psi \otimes dm_o\otimes db$ where $m_o$ is the $K$-invariant probability measure on $\cal F$ (cf. \cite[Sec. 3]{ELO}).  
Lee and Oh proved that each $\m_\psi^{\BR}$ is $MN$-ergodic
and its restrictions to $\pc$-minimal subsets of $\G\ba G$ yield all $N$-ergodic components (\cite[Thm. 10.1]{LO}, \cite[Thm. 1.3]{LO2}).
For $\mathsf{u}\in \inte\fa^+$, we consider the following directional recurrent set
$$\cal R_{\mathsf u}:=\{x\in \Ga\ba G: \limsup_{t\to +\infty}
x \exp(t{\mathsf u})\ne \emptyset\}.$$
Since $\mathsf u\in \inte\fa^+$, this is a $P$-invariant dense Borel subset of $\cal E$.

An immediate consequence of Theorem \ref{thm.Ano} is
the following:
\begin{cor}\label{bb2} For any $\psi\in D_\Ga^\star$ and $\mathsf{u}\in \inte \fa^+$, we have
\begin{enumerate}
\item If $\op{rank} G\le 3$ and $\mathsf{u}\in \inte\L_\Ga$ with $\psi({\mathsf u})=\psi_\Ga({\mathsf u})$, then $$\m^{\BR}_\psi(\G\ba G -\cal R_{\mathsf u})=0.$$

\item In all other cases,
$\m_\psi^{\BR}(\cal R_{\mathsf u})=0$.
\end{enumerate}
\end{cor}

This corollary is one of the main ingredients of the recent measure classification result \cite[Thm. 1.1]{LLLO}.

{\it{Added:}}  After we sent this paper to Andres Sambarino, he sent us a note giving a rough outline of a proof of the statement that for Anosov groups and for $\psi(\u)=\psi_\Ga(\u)$, the ergodicity of $(\Ga\ba G/M, \exp (t\u), \m_\psi)$  holds if and only if $\op{rank }\,G\le 3$, with a different approach based on the work of Guivarch.
\subsection*{Organization}
In section 2, we collect basic definitions. In section 3, we show that the set of directional conical limit points is either null or conull for any $(\Ga,\psi)$-conformal measure.
In section 4, we prove that the conservativity  of the Bowen-Margulis-Sullivan measure for one parameter diagonal flow implies its ergodicity, extending Hopf's argument.
In section 5, we relate the directional Poincare series with respect to $\psi$ and the correlation functions of the BMS measures and provide the proof of Theorem \ref{dio}.
In section 6, we specialize to Anosov groups and prove Theorem \ref{thm.Ano}.

\section{Preliminaries}\label{sec.pre}
Let $G$ be a connected, semisimple real algebraic group. We  decompose $\frak g=\op{Lie} G$ as 
$\frak g=\frak k\oplus\mathfrak{p}$, where $\frak k$ and $\frak p$ are the $+ 1$ and $-1$ eigenspaces of 
a Cartan involution $\theta$ of $\frak g$, respectively. 
We denote by $K$ the maximal compact subgroup of $G$ with Lie algebra $\frak k$, and by $X=G/K$ the associated symmetric space.
Choose a maximal abelian subalgebra $\frak a$ of $\frak p$ and
 a closed positive Weyl chamber $\frak a^+$ of $\frak a$. Set $A:=\exp \frak a$ and $A^+=\exp \frak a^+$. The centralizer of $A$ in $K$ is denoted by $M$. Consider the following pair of opposite maximal horospherical subgroups:
  $$N=N^-:=\{g\in G: a^{-n} g a^n\to e\text{ as $n\to +\infty$}\}\text{ and }$$
  $$ N^+:=\{g\in G: a^n g a^{-n}\to e\text{ as $n\to +\infty$}\}$$
  for any $a\in \inte A^+$; this definition is independent of the choice of $a\in \inte A^+$.
 
We set $$P=MAN,\quad\text{and} \quad P^+=MAN^+;$$ they are minimal parabolic subgroups of $G$ and $P\cap P^+=MA$.
The quotient space $\cal F=G/P$ is called the Furstenberg boundary of $G$, and via the Iwasawa decomposition $G=KP$,
$\cal F$ is isomorphic to $K/M$.

Let $\op N_K(\frak a)$  be the normalizer of $\frak a$ in $K$, and $\cal W:=\op N_K(\frak a)/M$ denote the Weyl group. Fixing a left $G$-invariant and right $K$-invariant Riemannian metric $d$ on $G$ induces
a Riemannian metric  on the associated symmetric space $X=G/K$, which we also denote by $d$ by abuse of notation. We denote by $\langle \cdot,\cdot\rangle$ and $\|\cdot \|$ 
the associated $\cal W$-invariant inner product and norm on $\frak a$.  

For $R>0$, set $A_R=\{a\in A:  \|\log a\|\le R\}$, $A_R^+=A_R\cap A^+$, and
$$G_R:=KA_R^+K.$$

\subsection*{$\fa$-valued Buseman functions}
The product map $K\times A\times N\to G$ is a diffeomorphism, yielding the well-known Iwasawa decomposition $G=KAN$.
 The Iwasawa cocycle $\sigma: G\times \cal F \to \frak a$ is defined as follows: for $(g, \xi)\in G\times \cal F$ with $\xi=[k]$ for $k\in K$,
$\exp \sigma(g,\xi)$ is the $A$-component of $g k$ in the $KAN$ decomposition, that is,
$$gk\in K \exp (\sigma(g, \xi)) N.$$

\begin{Def}\label{Bu} \rm The $\frak a$-valued Busemann function $\beta: \cal F\times G \times G \to\frak a $ is defined as follows: for $\xi\in \cal F$ and $g, h\in G$,
 $$\beta_\xi ( g, h):=\sigma (g^{-1}, \xi)-\sigma(h^{-1}, \xi).$$
\end{Def}

Denote by $w_0\in \cal W$ the unique element of $\cal W$ such that $\op{Ad}_{w_0}\frak a^+= -\frak a^+$. 
 \begin{Def}[Visual maps]\rm  
 For each $g\in G$, we define 
   $$g^+:=gP\in G/P\quad\text{and}\quad g^-:=gw_0P \in G/P.$$
Note that for $g\in G$, $g^{\pm}=g(e^{\pm})$.
\end{Def}

The opposition involution $\i:\frak a \to \frak a$ is
defined by \be\label{inv} \i ({\mathsf v})= -\op{Ad}_{w_0} ({\mathsf v}).\ee
When $G$ is a product of rank one groups, $\i$ is trivial.

 The set $\cal F^{(2)}=\{(g^+, g^-)\in \cal F\times \cal F: g\in G\}=G. (e^+, e^-)$ is the unique open $G$-orbit.
The $\frak a$-valued Gromov product on $\cal F^{(2)}$ is defined as follows:
for $(g^+, g^-) \in\cal F^{(2)}$,
$$
\cal G(g^+,g^-): = \beta_{g^+}(e,g)+\op i{\left(\beta_{g^-}(e,g)\right)}.
$$

\begin{lem}\label{lem.BPS}\cite[Prop. 8.12]{BPS}
There exist $c, c'>0$ such that for all $g\in G$,
$$
{c^{-1}}\norm{\cal G(g^+,g^-)}\leq d(o,gAo)\leq c\, \norm{\cal G(g^+,g^-)}+c'.
$$
\end{lem}

\begin{Def} [Cartan projection] \rm
For $g\in G$, there exists a unique element $\mu(g)\in \frak a^+$, called the Cartan projection of $g$, such that
\begin{equation*}
g\in K\exp(\mu(g))K.
\end{equation*}
\end{Def}

When $\mu(g)\in \inte \frak a^+$ and $g=k_1 \exp(\mu(g))k_2$,
we write $\kappa_1(g):=[k_1]\in K/M$ and
$\kappa_2(g):={k_2}\in M\ba K$, which are well-defined.

In the whole paper, we  fix the constant $d=d(G)\ge 2$ as in the following lemma.
\begin{lem}\cite[Lem. 5.8]{LO} \label{com} There exists $d\ge 2$ such that
for any $R>1$ and any $g\in G$,
$$\mu(G_R g G_R)\subset \mu(g)+ \fa_{dR}.$$
\end{lem}

%Let $\Pi$ denote the set of simple roots corresponding to $\frak a^+$.
 \begin{Def} \rm We say that a sequence
$g_i\to \infty$ regularly in $G$  if 
$\alpha(\mu(g_i))\to\infty$ as $i\to\infty$ for every simple root $\alpha$ corresponding to $\fa^+$.
 \end{Def}
 
 \begin{Def}\label{acc} \rm
\begin{enumerate}
\item A sequence $g_i\in G$ is said to converge to $\xi\in\cal F$, if $g_i\to\infty$ regularly in $G$ and $\lim\limits_{i\to\infty}\kappa_1(g_i)^+=\xi$.
\item A sequence $p_i=g_i(o) \in X$ is said to converge to $\xi\in \cal F$ if $g_i$ does. 
\end{enumerate}
\end{Def}

%In the rest of this section, we fix a discrete subgroup $\G<G$.

\begin{Def}[Limit set] For a Zariski dense discrete subgroup $\G< G$, we define the limit set $\La$ of $\Ga$ as follows: fixing $p\in X$,
$$
\La :=\left\{
\begin{array}{c}
\lim\limits_{i\to\infty} \ga_i p \in\cal F:  \ga_i\in\Ga
\end{array}
\right\}.
$$
By \cite[Lem. 2.13]{LO}, this definition is independent of the choice of $p\in X$ and coincides with one given by Benoist \cite[Def. 3.6]{Ben}; in particular, it is the unique $\Gamma$-minimal subset of $\F$.
\end{Def}

We later use the fact that $\La$ is a Zariski dense subset of $\F$ \cite[Lem. 3.6]{Ben}.

For any real-valued  functions $f(t)$ and $g(t)$, we write $f(t)\ll g(t) $ if there exists $C>0$ such that $f(t)\le C g(t)$ for all $t>1$. We write $f(t)\asymp g(t)$ if
$f(t)\ll g(t)$ and $g(t)\ll f(t)$.

\section{A zero-one law for $\nu_\psi(\La_{\mathsf u})$}
Let $\Ga<G$ be a Zariski dense discrete subgroup of $G$. Fix $\psi\in \fa^*$, and
a $(\Gamma, \psi)$-conformal measure $\nu_\psi$ on $\F$ as defined in \eqref{eq.conf}.

Recalling the notation $\La_{\mathsf u}$
from Definition \ref{conl}, the goal of this section is to prove the following dichotomy: 
\begin{prop}\label{lem.01} For any $\mathsf{u}\in \inte\fa^+$,
we have  $$\nu_\psi(\La_{\mathsf u})=1\quad\text{ or }\quad \nu_\psi(\La_{\mathsf u})=0.$$
\end{prop}

The proof of this proposition is based on the study of shadows.
\subsection*{Shadows}
For $p,q\in X=G/K$ and $r>0$, the shadow of the $r$-ball around $q$ as seen from $p$ is defined by
\begin{align*}
O_r(p,q)&:=\{g^+ \in \cal F : go=p,\, gA^+o\cap B(q,r)\neq\emptyset \},
\end{align*}
where $B(q,r)=\{x\in X: d(x, q)<r\}$.

Similarly, for $\xi\in\cal F$, we define the shadow of the $r$-ball around $q$ as seen from $\xi$ to be
$$
O_r(\xi,q):=\{g^+\in\cal F: g^-=\xi,\text{ }go\in B(q,r)\}.
$$

Note the following $G$-equivariance property: for any $g\in G$ and $r>0$,
\begin{align}\label{eq.eqv0}
gO_r(p,q)=O_r(gp,gq)\;\; \text{ and } \;\; gO_r(\xi,q)=O_r(g\xi,gq).
\end{align}

Note that for any $\xi\in \F$, $q\in X$ and $R>0$,
\be\label{oo} \bigcup_{r>R}O_r(\xi,q)=\{\eta\in\cal F :(\xi,\eta)\in\cal F^{(2)}\}.\ee

\begin{lem}\cite[Lem. 5.7]{LO}\label{lem.shadow2}
There exists $\kappa>0$ such that for any $r>0$ and $g\in G$, we have
$$
\sup_{\xi\in O_r(o,go)}\norm{\beta_\xi(e,g)-\mu(g)}\leq\kappa r.
$$
\end{lem}

The following lemma is an immediate consequence of \cite[Lem. 5.6]{LO}:
\begin{lem}\label{lem.F2}
For any $S>0$ and a sequence $g_i\to\infty$ regularly in $G$, we have,
for all sufficiently large $i$, the closure of
$O_S(o,g_io)\times O_S(g_io,o)$ is contained
in
$\cal F^{(2)}.$ 
\end{lem}

The following shadow lemma plays an important role in our paper. It was first presented in \cite[Thm. 3.3]{Alb} and then in \cite[Thm. 8.2]{Quint2} in slightly different forms.
\begin{lemma}[Shadow lemma] 
\cite[Lem. 7.8]{LO} \label{lem.shadow}
 There exists $S_0>0$ such that
$ c_1:=\inf_{\ga\in\Ga} \nu_\psi (O_{S}(\ga o,o))>0.$ Moreover,
 there exists $\kappa>0$ such that
 for all $S> S_0$,  and for all $\ga\in\Ga$,
$$
 c_1 e^{-\kappa \|\psi\| S} e^{-\psi(\mu(\ga))}\leq\nu_\psi(O_S(o,\ga o))\leq  e^{\kappa \|\psi\| S} e^{-\psi(\mu(\ga))}.
$$
 \end{lemma}

For any $R>0$, set
\be\label{GuR}
G_{{\mathsf u},R}:=\{g\in G : \norm{\mu(g)-t{\mathsf u}}<R\;\; \text{ for some }t\geq 0\}.\ee

\begin{lem}\label{lem.f} Let $R, S>0$.
If $g\in G_{{\mathsf u}, R}$, then
$$O_S(o, go)\subset \{ k^+\in \F: k\exp(t{\mathsf u})o\in B(g o,R+ 2d S) \text{ for some $t>0$}\}.$$
\end{lem}
\begin{proof}
For $\xi\in O_S(o,g o)$, there exist $k\in K$ and $a\in A^+$ such that $kao\in B(g o,S)$ and $\xi=k^+$.
It follows that $g^{-1}ka \in G_S$, and hence $\norm{\mu(g)- \log a}\leq dS $ by Lemma \ref{com}.
On the other hand, since $g\in  G_{{\mathsf u},R}$, there exists $t\geq 0$ such that $\norm{\mu(g)-t{\mathsf u}}< R$, and hence
\begin{align*} 
d(k\exp (t{\mathsf u}) o,go)&\leq d(k\exp (t{\mathsf u})o,ka o)+d(kao,go)\\
&<\norm{t{\mathsf u}-\log a}+S\le \|t{\mathsf u}-\mu(g)\|+\|\mu(g)-\log a\| +S\\ &\le  R+(d+1)S.
\end{align*}
This proves the lemma.
\end{proof}

The following Vitali-covering type lemma is a key ingredient of the proof of Proposition \ref{lem.01}.
\begin{lem}[Covering lemma] \label{cover} Fix $R>0$ and
consider $\{O_R(o, \ga o):\ga\in \G'\}$ for some infinite subset $\G'\subset \G_{\u, R}$.
There exists a subset $\G''\subset \G'$ such that 
$\{O_R(o, \ga o):\ga\in \G''\}$ consists of pairwise disjoint shadows and
\begin{equation}\label{eq.pR}
\mathop{\bigcup}_{\ga\in\Ga'}O_{R}(o,\ga o)\subset \mathop{\bigcup}_{\ga\in\Ga''} 
O_{10 dR}(o,\ga o).
\end{equation}
\end{lem}

\begin{proof}
Enumerate $\Ga'=\{\ga_i:i\in\bb N\}$ so that $\norm{\mu(\ga_i)}$ is nondecreasing.
Set $i_0=0$ and inductively define
$$
i_{n+1}:=\min\{i>i_n: O_{R}(o,\ga_i o)\bigcap \left(\cup_{j\leq n}O_{R}(o,\ga_{i_j}o)\right) =\emptyset\}.
$$
Set $\Ga'':=\{\ga_{i_n}:n\in\bb N\}$ so that $\{O_{R}(o,\ga o):\ga\in\Ga''\}$ consists of pairwise disjoint shadows.

For each $\ga_i\in \G'$, we claim that
$O_{R}(o,\ga_i o)\subset O_{10 dR}(o, \ga o)$ for some $\ga\in\Ga''$.
We may assume that $i_n<i<i_{n+1}$ for some $n$.
By definition of $i_{n+1}$, there exists $j\leq n$ such that $O_{R}(o,\ga_i o)\cap O_{R}(o,\ga_{i_j}o)\neq\emptyset$.
In particular, there exists $k_1\in K$, $a_i, a_{i_j}\in A^+$  such that $k_1a_io\in B(\ga_io,R)$ and $k_1a_{i_j}o\in B(\ga_{i_j}o,R)$.
Since $\ga_i^{-1}k_1 a_i, \ga_{i_j}^{-1}k_1a_{i_j}\in G_R$, we have 
$$
\norm{\mu(\ga_i)-\log a_i}\leq dR\text{ and }\norm{\mu(\ga_{i_j})-\log a_{i_j}}\leq dR
$$
by Lemma \ref{com}.
On the other hand, there exists $t_i, t_{i_j}\geq 0$ such that 
$$
\norm{\mu(\ga_i)-t_i\mathsf{u}}\leq R\text{ and }\norm{\mu(\ga_{i_j})-t_{i_j}\mathsf{u}}\leq R,
$$
as $\ga_i,\ga_{i_j}\in\Ga_{{\mathsf u},R}$.
Observe that
\begin{align*}
&\norm{\mu(\ga_i)}=d(o,\ga_io)\leq d(o,k_1a_{t_i}o)+d(k_1a_{t_i}o,k_1a_{i}o)+d(k_1a_{i}o,\ga_io)\\
&\leq  d(o,k_1a_{t_i}o)+dR+2R
=d(o,k_1a_{t_{i_j}}o)+dR+2R+(t_i-t_{i_j})\\
&\leq d(o,k_1a_{i_j}o)+2dR+3R+(t_i-t_{i_j})\leq d(o,\ga_{i_j}o)+2dR+4R+(t_i-t_{i_j})\\
&=\norm{\mu(\ga_{i_j})}+2dR+4R+(t_i-t_{i_j})\leq \norm{\mu(\ga_{i})}+2dR+4R+(t_i-t_{i_j}),
\end{align*}
and hence $t_i':=t_i+2dR+4R\ge t_{i_j}$.

Now let $k_2^+\in O_{R}(o,\ga_io)$ be arbitrary and $b\in A^+$ be such that $k_2bo\in B(\ga_io,R)$.
We have $\norm{\mu(\ga_i)-\log b}\leq dR$ by Lemma \ref{com}.
Since $\ga_i\in\Ga_{{\mathsf u},R}$, there exists $s\geq 0$ such that $\norm{\mu(\ga_i)-s\mathsf u}\leq R$.
Since 
\begin{align*}
&d(k_2a_so,k_1a_{t_i'}o)\leq d(k_2a_so, k_2bo)\\ &+d(k_2bo,\ga_i o)+d(\ga_i o,k_1a_{i}o)+d(k_1a_{i}o,k_1a_{t_i}o)+d(k_1a_{t_i}o,k_1a_{t_i'}o)\\
&\leq (dR+R)+R+R+(dR+R)+(2dR+4R)=4dR+8R,    
\end{align*}
there exists $0\leq s'\leq s$ such that $d(k_2a_{s'}o,k_1a_{t_{i_j}}o)\leq 4dR+8R$ by Lemma \ref{lem.mon} below.
Finally,
\begin{align*}
d(k_2a_{s'}o,\ga_{i_j}o)&< d(k_2a_{s'}o,k_1a_{t_{i_j}}o)+d(k_1a_{t_{i_j}}o,k_1a_{i_j}o)+d(k_1a_{i_j}o,\ga_{i_j}o)\\
&\leq (4dR+8R)+(dR+R)+R=5dR+10R,    
\end{align*}
which implies that $k_2^+\in O_{5dR+10R}(o,\ga_{i_j}o)\subset O_{10 dR}(o, \ga_{i_j} o)$, since $d\ge 2$. This finishes the proof.
\end{proof}

\begin{lem}\label{lem.mon}
Let $k_1,k_2\in K$, $t_1,t_2\geq 0$ be arbitrary.
For any $0\leq s_1\leq t_1$, there exists $0\leq s_2\leq t_2$ such that 
$$
d(k_1\exp(s_1\u)o,k_2\exp(s_2\u)o)\leq d(k_1\exp(t_1\u)o,k_2\exp(t_2\u)o).
$$
\end{lem}
\begin{proof}
This follows from the CAT(0) property of $G/K$ (cf. \cite{Eb}).
Consider the geodesic triangle $\triangle( pqr )$ in $G/K$ with vertices $p=o$, $q=k_1\exp(t_1\u)o$ and $r=k_1\exp(t_2\u)o$.
Let $\triangle( p'q'r')$ be the triangle in the Euclidean space which has the same corresponding side length to $\triangle pqr$.
Let $0\leq s_2\leq t_2$ be arbitrary and $r_1'$ be a point on the side $p'r'$ such that the segment $p'r_1'$ has length $\ell(p'r_1')=s_2$.
By a straightforward computation in Euclidean geometry, we can find a point $q_1'$ on the side $p'q'$ such that 
$$
\ell(q_1'r_1')\leq \ell(q'r')=\ell(qr)=d(k_1\exp(t_1\u)o,k_2\exp(t_2\u)o).
$$
Set $s_1:=\ell(p'q_1')$.
Since $G/K$ is a CAT(0) space, we get
$$
d(k_1\exp(s_1\u)o,k_2\exp(s_2\u)o)\leq \ell(q_1'r_1'),
$$
from which the lemma follows.
\end{proof}

We may write
$\La_{\mathsf u}=\mathop{\cup}_{R>0}\La_{\u,R}$ where 
\be\label{lau} 
\La_{\u,R}:=\bigcap_{m\geq1}\bigcup_{
\substack{
\ga\in \Ga_{{\mathsf u},R},\\
\|\mu(\ga)\|\geq m
}
}
O_{R}(o,\ga o),  \quad \text{ where $\Ga_{{\mathsf u}, R}:=\G\cap G_{{\mathsf u}, R}$.}
\ee

\begin{lem}\label{prop.LD}
If $R>1$ is large enough, 
for any $f\in L^1(\nu_\psi)$ and for $\nu_\psi$-a.e. $\xi\in\La_{\u,R}$, we have
$$
\lim_{i\to \infty} \frac{1}{\nu_\psi(O_{R}(o,\ga_i o))}\int_{O_{R}(o,\ga_i o)}f\,d\nu_\psi= f(\xi)
$$
 for any sequence $\ga_i\to\infty$ in $\Ga_{{\mathsf u},R}$ such that $\xi\in O_{R}(o,\ga_i o)$.
\end{lem}

We define a maximal operator $M_R$ on $L^1(\nu_\psi)$ as follows: for all $f\in L^1(\nu_\psi)$ and all $\xi\in\La_{\u,R}$, set
$$
M_Rf(\xi):=\limsup_{\substack{\ga\in \Ga_{{\mathsf u}, R}, 
\|\mu(\ga)\|\to\infty,\\
 \xi\in O_{R}(o,\ga o)}
} \frac{1}{\nu_\psi(O_{R}(o,\ga o))}\int_{O_{R}(o,\ga o)}f\,d\nu_\psi;$$
this is well-defined by the definition of $\La_{\u, R}$.

Note that Lemma \ref{prop.LD} holds trivially for $f\in C(\La)$.
Once the weak type inequality for the maximal functions is established as in Lemma \ref{lem.MI}, Lemma \ref{prop.LD}  follows from a standard argument using the density of $C(\La)$ in $L^1(\nu_\psi)$.
\begin{lem}\label{lem.MI}
If $R>1$ is large enough, then $M_R$ is of weak type $(1,1)$; for all $f\in L^1(\nu_\psi)$ and $\lambda>0$, we have
$$
\nu_\psi(\{\xi\in\La_{\mathsf u, R}:|M_Rf(\xi)|>\lambda\})\ll\frac{1}{\lambda}\norm{f}_{L^1(\nu_\psi)}
$$ where the implied constant is independent of $f$.
\end{lem}
\begin{proof} Let $R>1$ be large enough to satisfy Lemma \ref{lem.shadow}.
Let $\lambda>0$ be arbitrary.
By definition of $M_R$, there exists an infinite subset $\Ga'\subset\Ga_{{\mathsf u},R}$ such that
\begin{align*}
\{\xi\in\La_{\mathsf u, R} :|M_Rf(\xi)|>\lambda\}\subset \bigcup_{\ga\in\Ga'} O_{R}(o,\ga o),\text{ and}\\
\frac{1}{\nu_\psi(O_{R}(o,\ga o))}\int_{O_{R}(o,\ga o)}f\,d\nu_\psi>\lambda \text{ for all }\ga\in \Ga'.
\end{align*}
By Lemma \ref{cover},
there exists $\G''\subset \G'$ so that $\{O_{R}(o,\ga o):\ga\in\Ga''\}$ consists of pairwise disjoint shadows and 
\begin{equation}\label{eq.5R}
\mathop{\bigcup}_{\ga\in\Ga'}O_{R}(o,\ga o)\subset \mathop{\bigcup}_{\ga\in\Ga''} 
O_{10 dR}(o,\ga o).
\end{equation}
Hence, by Lemma \ref{lem.shadow},
\begin{align*}
&\nu_\psi(\{\xi\in\La_{\mathsf u, R}:|M_Rf(\xi)|>\lambda\})\leq \nu_\psi(\mathop{\bigcup}_{\ga\in\Ga'} O_{R}(o,\ga o))\\
&
\leq \nu_\psi(\mathop{\bigcup}_{\ga\in\Ga''} O_{10 dR}(o,\ga o)) \le  \sum_{\ga\in \G''}\nu_\psi(O_{10 dR}(o,\ga o))
\\ &\asymp \sum_{\ga\in \G''}\nu_\psi(O_{R}(o,\ga o)) \leq\frac{1}{\lambda}\int_{\mathop{\cup}_{\ga\in\Ga''} O_{R}(o,\ga o)}f\, d\nu_\psi\leq \frac{1}{\lambda}\norm{f}_{L^1(\nu_\psi)}.
\end{align*}
\end{proof}

\subsection*{Proof of Proposition \ref{lem.01}} Let $R>1$ be large enough to satisfy Lemma \ref{lem.MI}. Suppose that $\nu_\psi(\La_{\mathsf u})>0$.
Then for all sufficiently large $R>1$, we have $\nu_\psi(\La_{\u,R})>0$.
By applying Lemma \ref{prop.LD} with $f=\mathbbm{1}_{\La_{\mathsf u}^c}$, there exists $\xi\in \La_{\u,R}$, we obtain a sequence $\ga_i\in\Ga$ such that $\xi\in O_{R}(o,\ga_i o)$ and 
$$
\lim_{i\to \infty} \frac{\nu_\psi(O_{R}(o,\ga_i o)\cap\La_{\mathsf u}^c)}{\nu_\psi(O_{R}(o,\ga_i o))} = 0.
$$
Since $\nu_\psi(O_{R}(o,\ga_i o))\asymp e^{-\psi(\mu(\ga_i))}$ by Lemma \ref{lem.shadow},
\be\label{zzz}
\lim_{i\to \infty} e^{\psi(\mu(\ga_i))} {\nu_\psi(O_{R}(o,\ga_i o)\cap\La_{\mathsf u}^c)} = 0.
\ee
By Lemma \ref{lem.shadow2},
\begin{align*}
\nu_\psi(O_{R}(o,\ga_i o)\cap\La_{\mathsf u}^c)&=\int \mathbbm{1}_{O_{R}(o,\ga_i o)\cap\La_{\mathsf u}^c}(\xi)\,d\nu_\psi(\xi)\\
&=\int \mathbbm{1}_{O_{R}(\ga_i^{-1}o, o)\cap\La_{\mathsf u}^c}(\xi)e^{\psi(\beta_\xi(e,\ga_i^{-1}))}\,d\nu_\psi(\xi)\\
&\asymp e^{-\psi(\mu(\ga_i))}\nu_\psi(O_{R}(\ga_i^{-1}o, o)\cap\La_{\mathsf u}^c).
\end{align*}

Hence as $i\to \infty$,
$$\nu_\psi(O_{R}(\ga_i^{-1}o, o)\cap\La_{\mathsf u}^c)\asymp e^{\psi(\mu(\ga_i))}
\nu_\psi(O_{R}(o,\ga_i o)\cap\La_{\mathsf u}^c)\to 0.$$

Passing to a subsequence, we may assume that $\ga_i^{-1}o$ converges to
some $\eta_0\in\La$.
By \cite[Lem. 5.6]{LO}, for all sufficiently large $i$,
$$\nu_\psi(O_{R/2}(\eta_0,o)\cap\La_{\mathsf u}^c)\leq\nu_\psi(O_R(\ga_i^{-1}o,o)\cap\La_{\mathsf u}^c).$$  Therefore
$$\nu_\psi(O_{R/2}(\eta_0,o)\cap\La_{\mathsf u}^c)= 0.$$
Since $R>1$ is an arbitrary large number, varying $R$, we get from \eqref{oo} that
\begin{equation}\label{eq.cell}
\nu_\psi(\La_{\mathsf u}^c\cap \{\eta\in\cal F:(\eta,\eta_0)\in\cal F^{(2)}\})=0.
\end{equation}

We now claim that for any $\eta\in\La_{\mathsf u}^c$, there exists a neighborhood $U_\eta$ of $\eta$ such that $\nu_\psi(\La_{\mathsf u}^c\cap U_\eta)=0$.
If $(\eta,\eta_0)\in\cal F^{(2)}$, this is immediate from \eqref{eq.cell}.
Otherwise, by the Zariski density of $\Ga$ and the fact that $\La$ is the unique $\Ga$-minimal subset of $\F$, we can find $\ga\in\Ga$ such that $(\ga\eta,\eta_0)\in\cal F^{(2)}$.
The claim follows again from \eqref{eq.cell}, since $\nu_\psi$ is $\Ga$-quasi-invariant.
This finishes the proof.$\qed$

\section{Hopf's argument for higher rank cases}\label{sec.TG}
Let $\Ga<G$ be a Zariski dense discrete subgroup. 
We fix $\psi\in \fa^*$ and a pair $(\nu_\psi, \nu_{\psi \circ \i})$ of $(\Gamma, \psi)$ and $(\Gamma, \psi\circ \i)$-conformal measures on $\F$ respectively.

\begin{Def}[Hopf parametrization of $G/M$] \rm
The map
$$gM\mapsto (g^+, g^-, b=\beta_{g^+}(e, g))$$
gives a homeomorphism between $G/M$ and $\cal F^{(2)}\times \frak a$.
\end{Def}

 \subsection*{Bowen-Margulis-Sullivan measures.} 
Define the following $A$-invariant Radon measure $\tilde{\m}=\tilde{\mathsf m} (\nu_{\psi}, \nu_{\psi\circ \i})$ on $G/M$ 
  as follows: for $g=(g^+, g^-, b)\in \cal F^{(2)}\times \frak a$,
\begin{equation*} \label{eq.BMS1}
d \tilde{\m} (g)= e^{\psi(\cal G(g^+,g^-))}\;  d\nu_{\psi} (g^+) d\nu_{\psi\circ \i}(g^-) db
\end{equation*}
  where $db $ is the Lebesgue measure on $\frak a$.
\noindent
We note that this is a non-zero measure; otherwise,
$\nu_\psi$ is supported on a proper Zariski subvariety of $\cal F$ by Fubini's theorem, but since $\G$ is Zariski dense and $\nu_\psi$ is $\Ga$-conformal, that is not possible. 
The measure $\tilde{\m}$ is left $\Gamma$-invariant, and hence induces a measure on $\Gamma\ba G/M$, which we denote by $\m$.

We fix $\mathsf{u}\in \inte \fa^+$ and set for all $t\in \br$,
$$a_t:=\exp t{\mathsf u}.$$

Recall the following definitions:
\begin{enumerate}
\item
A Borel subset $B\subset \Ga\ba G/M$ is called a \textit{wandering set} for $\m$ if for $\m$-a.e. $x\in B$, we have $\int_{-\infty}^\infty \mathbbm{1}_B(xa_t )\,dt<\infty$.
\item
We say that $(\Ga\ba G/M, \m,\{a_t\})$ is \textit{conservative} if there is no wandering set $B\subset \Ga\ba G/M$ with $\m(B)>0$.
\item
We say that $(\Ga\ba G/M, \m,\{a_t\})$ is \textit{completely dissipative} if  $\Ga\ba G/M$ is a countable union of wandering sets modulo $\m$.
\end{enumerate}

\begin{prop}\label{lem.dich}
The flow
 $(\Ga\ba G/M, \m,\{a_t=\exp (t\mathsf u)\})$ is conservative (resp. completely dissipative) if and only if
 $\max (\nu_\psi(\La_{\mathsf u}), \nu_{\psi\circ \i} (\La_{\i({\mathsf u})})) >0$ (resp.
 $\nu_\psi(\La_{\mathsf u})=0= \nu_{\psi\circ \i} (\La_{\i({\mathsf u})})$). 
\end{prop}
\begin{proof}
Suppose that $(\Ga\ba G/M, \m,\{a_t\})$ is conservative.
Let $B$ be a compact subset of $\Ga\ba G/M$ with $\m (B)>0$.
If we set $B_0^{\pm}:=\{x\in B: \limsup_{t\to \pm \infty} xa_{t}\cap B\ne\emptyset\}$, then $\m (B_0^+\cup B_0^-)>0$.
Since $\tilde \m$ is equivalent to $\nu_\psi \otimes \nu_{\psi \circ \i} \otimes db$, it follows that $\m (B_0^+)>0$ (resp. $\m (B_0^-)>0$) if and only if $\nu_\psi (\La_{\mathsf u})>0$ (resp. 
$\nu_{\psi\circ \i} (\La_{\i ({\mathsf u})})>0$). Hence $\max (\nu_\psi(\La_{\mathsf u}), \nu_{\psi\circ \i} (\La_{\i({\mathsf u})})) >0$.

Now suppose that  $\nu_\psi(\La_{\mathsf u})>0$ (resp. $\nu_{\psi\circ \i} (\La_{\i({\mathsf u})}) >0$).
Then by Proposition \ref{lem.01}, $\nu_\psi(\La_{\mathsf u})=1$ (resp. $\nu_{\psi\circ \i} (\La_{\i({\mathsf u})}) =1$.) Hence for $\m$ a.e. $[g]$, we have $g^+\in \La_{\mathsf u}$ (resp. $g^-\in \La_{\i({\mathsf u})}$), and hence $[g] a_{t_i}$ is convergent for some sequence $t_i\to \pm \infty$. It follows that for $\m$ a.e. $x$, there exists a compact subset $B$ such that $\int_{\br} \mathbbm 1_B(xa_t) dt=\infty$. 
We claim that this implies that $(\Ga\ba G/M, \m,\{a_t\})$ is conservative. Assume in contradiction that
there exists a wandering set $W \subset \Ga\ba G/M$  with
$0<\m (W)<\infty$. 
By the $\sigma$-compactness of $\Gamma\ba G/M$, there exists
a compact subset $B$
such that \be\label{mw} \m \{x\in W:\int_{\br} \mathbbm 1_B(xa_t) dt=\infty\}\ge \m (W)/2.\ee

On the other hand, there exists an integer $n \geq 1$ for which the set
$$ W_n:=\left\{w \in W: \int_{-\infty}^\infty \mathbbm{1}_{W}(wa_t )\,dt\le n\right\}$$
has $\m$-measure strictly bigger than $\m(W)/2$. Note that the set $E:=W_n \exp (\bb R \v)$ is $\{a_t\}$-invariant and any $w\in E$ satisfies $\int_{-\infty}^\infty \mathbbm{1}_{W}(wa_t )\,dt\le n$.
 Hence, for any $R>0$, we get
\begin{align*}
    &\int_{W_n} \int_{-R}^R \mathbbm{1}_B (wa_t)dt d\m = \int_{-R}^R \int_{W_n} \mathbbm{1}_B (wa_t)d\m dt \\& 
    = 
    \int_{-R}^R \m(Ba_{-t} \cap W_n) dt =\int_{-R}^R \m(B \cap W_n a_t) dt \\& = \int_{-R}^R \m((B\cap E) \cap W_n a_t) dt =\int_{-R}^R \int_{B \cap E} \mathbbm{1}_{W_n} (xa_{-t}) d\m dt \\& = \int_{B \cap E} \int_{-R}^R \mathbbm{1}_{W_n} (xa_{-t})dt d\m  \le \int_{B \cap E} \int_{\bb R} \mathbbm{1}_{W_n} (xa_{-t})dt d\m \\ &\le \int_{B \cap E} n \;dt = n \cdot \m(B \cap E) < \infty
\end{align*}
where finiteness follows from the fact that $B$ is compact and $\m$ is Radon. 
Hence  $\int_{W_n} \int_{\br} \mathbbm{1}_B (wa_t)dt d\m <\infty$; so
$$\m \{ x\in W: \int_{\br} \mathbbm{1}_B (wa_t)dt  <\infty\}\ge \m (W_n)>\m(W)/2 .$$
contradicting \eqref{mw}.
 The rest of the claims can be proven similarly. 
\end{proof}

 Let $\tilde{\m}'$ denote the $M$-invariant lift of $\tilde\m$ to $G$ and ${\m}'$ the measure on $\G\ba G$ induced by $\tilde{\m}'$.
Since $\G$ is Zariski dense, there exists a normal subgroup
$M_\Ga<M$ of finite index such that each $\pc$-minimal subset of $\Ga\ba G$ is $M_\Ga$-invariant and the collection of all $\pc$-minimal subsets is parameterized by $M/M_\Ga$ 
(\cite[Thm. 1.9 and 2]{GR}, see also \cite[Sec. 3]{LO2}).

We will need the following notion:
 \begin{Def}[Transitivity group] \label{ttt}  For $g \in G$ with $g^{\pm}\in \La$,  define the subset $\mathcal H_\Gamma^s(g)<AM$ as follows:
 $am \in \mathcal{H}^s_{\G}(g)$ if and only if
    there exist $\gamma\in \Ga$ and a sequence $h_i\in N^-\cup N^+$, $i=1,\ldots, k$ such that
 $$(gh_1h_2\ldots h_r)^{\pm} \in \La \,\,\,\text{for all}\,\,\, 1\leq r\leq k\quad \text{and}\quad 
\gamma g h_1h_2\ldots h_k=gam. $$
It is not hard to check that $\mathcal H_\Gamma^s(g)$ is a subgroup (cf. \cite[Lem. 3.1]{Win}); it is called the strong transitivity subgroup.
\end{Def}

The following was obtained in \cite{LO2} using the work of Guivarch-Raugi \cite[Thm. 1.9]{GR}.
\begin{lem}\cite[Coro. 3.8]{LO2}\label{st}
For any $g\in G$ with $g^{\pm}\in \La$, the closure of $\mathcal H_\Gamma^s(g)$ contains $ AM_\Ga $.
\end{lem}

We now prove the following higher rank version of the Hopf-dichotomy, using Lemma \ref{st}.
\begin{prop}\label{mini}
Let $Y$ be a $\pc$-minimal subset of $\G\ba G$ such that $\m'(Y)>0$.
Then
$(\m'|_Y, \{a_t\})$ is conservative if and only if
$( \m'|_Y, \{a_t\})$ is ergodic.
\end{prop}

\begin{proof} 
Suppose that $(\m'|_Y, \{a_t\})$ is conservative. Fix $x_0\in \op{supp} ({\m}'|_Y)$ and let $B_n\subset \Ga\ba G$ denote the ball of radius $n$ centered at $x_0$.
Let $r$ be a positive function on $[0,\infty)$
which is affine on each $[n, n+1]$ and $r(n)=1/(2^{n+1}  {\m}'(B_{n+1}))$.
Then the function $\rho(x):= r( d(x_0, x))$ is a positive Lipschitz function on $\Ga\ba G$  with a uniform Lipschitz constant.
In particular, it is uniformly continuous and $\rho\in L^1( \m')$, since
$$
\norm{\rho}_{L^1( \m')}=\sum_{n=1}^\infty \int_{B_{n}-B_{n-1}}\rho\, d \m'\le \sum_{n=1}^\infty\frac{1}{2^{n} \m'(B_{n})} \m'(B_n)<\infty.
$$

By the definition of $\rho$, for all $[g]\in \Ga\ba G$ such that $g^+\in\La_{\mathsf u}$ and $g^-\in\La_{\op{i}({\mathsf u})}$, we have
\begin{equation}\label{eq.div}
\int_0^\infty \rho([g]a_t)\,dt=\int_0^\infty \rho([g]a_{-t})\,dt=\infty.
\end{equation}
Now let $f\in C_c(\Ga\ba G)$ be arbitrary. By the Hopf ratio ergodic theorem,  the following $f_+$ and $f_-$ are well-defined and equal $\m'$-a.e.:
\begin{equation}\label{eq.ratio}
f_+(x):=\lim_{T\to { \infty}}\frac{\int_0^T f (xa_t)\,dt}{\int_0^T \rho(xa_t)\,dt}
\text{ and } f_-(x):=\lim_{T\to { \infty}}\frac{\int_0^T f (xa_{-t})\,dt}{\int_0^T \rho(xa_{-t})\,dt}.
\end{equation}
By the uniform continuity of $f$ and $\rho$, \eqref{eq.div} and the assumption that $\mathsf{u}\in \inte \fa^+$,  $f_{\pm}$ are $N^{\pm}$-invariant respectively. Let $\tilde f_{\pm}:G\to\br$ be
a left $\Gamma$-invariant lift of $f_{\pm}$. Let $\cal B$ denote the Borel $\sigma$-algebra of $G$ and set
$$\Sigma:=\{B\in \cal B: \tilde{\mathsf m}'(B\triangle B^\pm)=0
\text{  for some $B^\pm\in \cal B$ such that $ \Ga B^{\pm}= B^{\pm} N^\pm$ }\}.$$

Let $F:G\to \br$ be a $\Sigma$-measurable and left $\Gamma$-invariant function such that
$F(g)=\tilde f_+(g)=\tilde f_-(g) $ for $\tilde \m'$ a.e $g \in G$. Set
$$
E:=\left\{ gAM :
\begin{array}{c}
F|_{gAM}\text{ is measurable and}\\
F(gam)=\tilde f_+(gam)=\tilde f_-(gam)\\
\text{for Haar a.e. }am\in AM
\end{array}
\right\}\subset G/AM.
$$
By Fubini's theorem, $E$ has full measure in $G/AM\simeq\cal F^{(2)}$ with respect to the measure $d\nu_{\psi}\otimes d\nu_{\psi\circ \i}$.
For all small $\e>0$, define functions $F^\e, \tilde f_\pm^\e : G\to\bb R$ by
$$
F^\e(g):=\tfrac{1}{\op{Vol}(AM)_\e}\int_{(AM)_\e}F(g\ell)\,d\ell,\text{ }\tilde f_\pm^\e(g):=\tfrac{1}{\op{Vol}(AM)_\e} \int_{(AM)_\e}\tilde f_\pm(g\ell)\,d\ell
$$where $(AM)_\e$ denotes the $\e$-ball around $e$ in $AM$ and $d\ell$ is the Haar measure on $AM$.
Note that if $gAM\in E$, then $F^\e$ and $ \tilde f_\pm^\e$ are continuous and identical on $gAM$.
Moreover, $F^\e$ is left $\Ga$-invariant and $\tilde f_\pm ^\e$ are $N^\pm$-invariant, as $AM$ normalizes $N^{\pm}$.
Using the isomorphism between $G/AM$ and $\cal F^{(2)}$ given by $gAM\mapsto (g^+, g^-)$,
we may consider $E$ as a subset of  $\cal F^{(2)}$.
We then define
\begin{align*}
E^+:&=\{\xi\in\La : (\xi,\eta') \in E\;\;\text{ for $\nu_{\psi\circ \i}$-a.e. }\eta'\in\La\};\\
E^-:&=\{\eta\in\La : (\xi',\eta) \in  E\;\; \text{ for $\nu_{\psi}$-a.e. }\xi'\in\La\}.
\end{align*}
 
Then $E^+$ is $\nu_\psi$-conull and $E^-$ is $\nu_{\psi\circ \i}$-conull by Fubini's theorem.
By a similar argument as in \cite[Lem. 4.6]{LO2},  we can show that
for any $gAM\in E$ with $g^{\pm}\in E^{\pm}$, and any $\e>0$,
$F^\epsilon|_{gAM}$ is $AM_\Gamma$-invariant, using the fact that the closure of
$\mathcal H_\Ga^s(g)$ contains $AM_\Ga$ (Lemma \ref{st}). 
It follows that $F$ is $\Sigma_0$-measurable where 
$$\Sigma_0:=\{B\in\cal B:  B=\Ga BAM_\Ga\}.$$  

We claim that if $f$ is $M$-invariant, then $F$ is constant on the $\m'$-conull set
$E^\dagger:=\{g\in G: g^{\pm}\in E^{\pm}\}$. Using Hopf's ratio ergodic theorem once more, this would in turn imply that
$\m'$ is $M\{a_t\}$-ergodic.
Assume $f$ is $M$-invariant. Since $F=\lim_{\e\to 0} F^\e$ $\m'$-a.e. by the Lebesgue differentiation theorem, it suffices to show that  $F^\e$ is constant on $E^\dagger$.
Since $F^\e$ is  $AM$-invariant on $E^{\dagger}$ and $F^{\e} (g h)=F^\e(g)$ for all $g\in E^\dagger$ and $h\in N^{\pm}$ with $gh\in E^\dagger$,
it is again enough to show  that for any $g_1, g_2\in E^\dagger$, there exist $h_1,h_2, h_3\in N^+\cup N^-$ such that $g_1 h_1 h_2h_3 \in g_2 AM$ and $g_1h_1, g_1h_1h_2\in E^\dagger$.

We note that if $(\xi, \eta_1), (\xi, \eta_2)\in \F^{(2)}$, then there exist $g\in G, h\in N$ such that $(\xi, \eta_1)=(g^+, g^-)$
and $(\xi, \eta_2)=((gh)^+, (gh)^-)$. Similarly, 
 if $(\xi_1, \eta), (\xi_2, \eta)\in \F^{(2)}$, then there exist $g\in G, h\in N^+$ such that $(\xi_1, \eta)=(g^+, g^-)$
and $(\xi_2, \eta)=((gh)^+, (gh)^-)$.

Note that $E^+$ is $\Gamma$-invariant. Since the limit set $\La$ is the unique $\Gamma$-minimal subset of $\F$, the closure of $E^+$ contains $\La$, and in particular it is Zariski dense.
Therefore we can choose $\xi\in E^+$ such that
$(\xi, g_1^-), (\xi, g_2^-)\in \F^{(2)}$. Let  $h_1, h_2, h_3\in N^+\cup N^-$ be such that 
\begin{align*}
    (\xi, g_1^-)&=(g_1h_1^+, g_1h_1^-)\\
    (\xi, g_2^-)&= (g_1h_1h_2^+, g_1h_1h_2^-)\\
    (g_2^+, g_2^-)&= (g_1 h_1h_2h_3^+, g_1h_1h_2h_3^-).
\end{align*}
Hence the claim is proved.
 In particular, $\m'$ is $AM$-ergodic.
 
 Let $\tilde Y\subset G$ be the $\Gamma$-invariant lift of $Y$.
 In order to show that $\m'|_Y$ is $\{a_t\}$-ergodic, it suffices to show that $F$, associated to an arbitrary function $f\in C_c(\Ga\ba G)$, is constant on $\tilde Y$.
 It follows from the $AM$-ergodicity of
 $\m'$ that $\Sigma_0$ is $\tilde \m '$-equivalent to a finite $\sigma$-algebra generated by $\{B.s: s\in M_\Gamma\ba M\}$ for some $B\in \Sigma_0$.  Since $\{\tilde Y.s: s\in M_\Ga\ba M\}\subset \Sigma_0$ and the $\tilde Y.s$'s are mutually disjoint, it follows that $\tilde Y= B.s$ mod $\tilde \m'$  for some $s\in M_\Gamma\ba M$. 

Since $F$ is constant on $B.s$, being $\Sigma_0$-measurable, it implies that $F$ is constant on $\tilde Y$, concluding that $\m'|_Y$ is $\{a_t\}$-ergodic. 

Now to show the converse, assume that $(\mathsf m'|_Y, \{a_t\})$
is ergodic. Since the quotient map $\Ga\ba G\to \Ga\ba G/M$ is a proper map,
it suffices to show that
$(\Ga\ba G/M, \mathsf m, \{a_t\})$ is conservative when it is ergodic.
Assume that $(\Ga\ba G/M, \mathsf m, \{a_t\})$ is ergodic. Then it is either conservative or completely dissipative by the Hopf decomposition theorem \cite{Kr}.
Suppose it is completely dissipative. Then
it is isomorphic to a translation on $\bb R$ with respect to the Lebesgue measure.
This implies that the dimension of $\fa$ must be one, since
$\tilde{\m}=\tilde{\mathsf m} (\nu_{\psi}, \nu_{\psi\circ \i})$ gives measure zero on any one dimensional flow otherwise.
It also implies that
$\nu_\psi\otimes \nu_{\psi \circ \i}$ is supported on a single $\Gamma$-orbit, say, $\Ga(\xi_0, \eta_0)$ in $\cal F^{(2)}$.
Since $\nu_\psi$ (resp. $\nu_{\psi\circ\i}$) must be an atomic measure supported on $\Ga\xi_0$ (resp. $\Ga\eta_0$), 
it follows that $(\Ga\xi_0\times\Ga\eta_0) \cap \F^{(2)} = \Ga(\xi_0,\eta_0)$.
This implies that
 $\Ga \eta_0\subset \G_{\xi_0} \eta_0$ where $\G_{\xi_0}$ denotes the stabilizer of $\xi_0$ in $\Ga$. Since the limit set of $\G_{\xi_0}$ is finite (as we are in the rank one situation), this is a contradiction as $\Gamma$ is non-elementary.
 This proves that $\mathsf m$ is conservative for the $\{a_t\}$-action.
\end{proof}

\section{Directional Poincare series}\label{sec:DP}
Let $\Ga<G$ be a Zariski dense discrete subgroup.
We define the limit cone $\L_\Gamma\subset \fa^+$ as the  asymptotic cone of $\mu(\Ga)$.
Then $\L_\Ga$ coincides with the smallest cone containing the Jordan projection of $\G$
%, as defined by Benoist (see \cite{ELO}) and in particular, $\L_\Ga$
and is a convex cone with non-empty interior \cite{Ben}. 

Quint \cite{Quint1} introduced the following:
\begin{Def}\label{def.GI}\rm The growth indicator function $\psi_{\Gamma}\,:\,\frak a^+ \rightarrow \br \cup\lbrace- \infty\rbrace$  is defined as a homogeneous function, i.e., $\psi_\Gamma (t{\mathsf u})=t\psi_\Gamma ({\mathsf u})$ for all $t>0$, such that
  for any unit vector $\mathsf{u}\in \frak a^+$,
 \begin{equation*}
\psi_{\Gamma}(\mathsf u):=\inf_{\underset{\mathsf{u}\in\cal C}{\mathrm{open\;cones\;}\cal C\subset \fa^+}}\tau_{\cal C}
%\inf_{\underset{u \in\scrC}{\mathrm{open\;cones\;}\scrC\subset \fa^+}} \limsup_{t\to \infty} \frac{1}{t} {\log\#\lbrace \gamma\in \Gamma\,:\, \mu(\ga) \in\scrC\;,  \|\mu(\gamma)\|\leq t\rbrace}.
\end{equation*}
where $\tau_{\cal C}$ is the abscissa of convergence of the series $\sum_{\ga\in\Ga, \mu(\ga)\in\cal C}e^{-t\norm{\mu(\ga)}}$.

% \begin{equation*}
%\psi_{\Gamma}({\mathsf u}):=\inf_{\underset{\mathsf u \in\cal C}{\mathrm{open\;cones\;}\cal C\subset \frak a^+}} \limsup_{t\to \infty} \frac{1}{t} {\log\#\lbrace \gamma\in \Gamma\,:\, \mu(\ga) \in\cal C\;,  \|\mu(\gamma)\|\leq t\rbrace}.
%\end{equation*}
We consider $\psi_\G$ as a function on $\fa$ by setting $\psi_\Ga=-\infty$ outside $\fa^+$.
\end{Def}
Quint showed that $\psi_\Ga$ is upper semi-continuous, $\psi_\G >0$ on $\inte \L_\Ga$, 
$\psi_\Ga\ge 0$ on $\L_\Ga$ and $\psi_\Ga = -\infty$ outside $\L_\Ga$ \cite[Thm. IV.2.2]{Quint1}.

\begin{lemma}\label{lem.neq} 
Let $\psi\in \fa^*$ and $\mathsf{u}\in \inte \fa^+$ be such that $\psi({\mathsf u})>\psi_\Ga({\mathsf u})$. Then for any $R>0$,
$$
\sum_{\ga\in\Ga_{{\mathsf u},R}}e^{-\psi(\mu(\ga))}<\infty.
$$
\end{lemma}
\begin{proof} 
 Since $\psi({\mathsf u}) > \psi_\Ga ({\mathsf u})$, the upper-semi continuity of $\psi_\Ga$ implies that there exists a small open convex cone $\cal C$ containing $\mathsf{u}$ such that $\overline{\cal C}\subset \inte \fa^+$ and
$ \psi >\psi_\Ga$ on $\overline{\cal C}$. Since $\psi>\psi_\Ga$ on some open convex cone $\cal C'$ containing $\overline{\cal C}$,
we can choose 
a continuous homogeneous function $\theta:\fa \to \br$ such that
$\psi \ge \theta >\psi_\Ga  $ on $\cal C$ and $\theta >\psi_\Ga$ on $\fa^+$.
Since $\psi_\Ga =-\infty$ outside $\fa^+$,
we have $\theta >\psi_\Ga $ on $\fa-\{0\}$.
Applying \cite[Lem. III.1.3]{Quint1} to the measure $\sum_{\ga\in\Ga}\delta_{\mu(\ga)}$ on $\fa^+$, we get
$$
\sum_{\ga\in\Ga,\mu(\ga)\in\cal C}e^{-\psi(\mu(\ga))}\leq \sum_{\ga\in\Ga}e^{-\theta(\mu(\ga))}<\infty,
$$

Since $\#\{\ga\in\Ga_{{\mathsf u},R}:\mu(\ga)\not\in\cal C\}<\infty$ for any $R>0$, 
 the lemma follows.
\end{proof}

Let $\psi\in \fa^*$ and fix a pair of $(\Gamma, \psi)$ and $(\Gamma, \psi\circ \i)$-conformal measures ($\nu_\psi$, $\nu_{\psi \circ \i}$) on $\F$ respectively.
We let $\m$ denote the BMS measure on $\Ga\ba G/M$ associated to $(\nu_\psi, \nu_{\psi \circ \i})$.

We fix a unit vector $\mathsf{u}\in \inte \fa^+$ such that $\psi({\mathsf u})>0$, and 
set $$a_t:=\exp(t{\mathsf u})\quad\text{ and }\quad \delta:=\psi({\mathsf u}).$$
For an interval $I\subset \br$, we sometimes write $a_I=\{a_t: t\in I\}$.
We make the following simple observation: for any $R>0$,
\be\label{same} \sum_{\ga\in\Ga_{\i({\mathsf u}),R}}e^{-\psi(\i (\mu(\ga)))}=
\sum_{\ga^{-1} \in\Ga_{{\mathsf u},R}} e^{-\psi(\mu(\ga^{-1}))}=\sum_{\ga \in\Ga_{{\mathsf u},R}} e^{-\psi(\mu(\ga))}.\ee

\begin{lemma}\label{lem.se}
If $\max (\nu_\psi(\La_{\mathsf u}), \nu_{\psi\circ \i} (\La_{\i ({\mathsf u})})) >0$, then there exists $R>0$ such that
$$
\sum_{\ga\in\Ga_{{\mathsf u},R}}e^{-\psi(\mu(\ga))}=\infty= \sum_{\ga\in\Ga_{\i({\mathsf u}),R}}e^{-\psi(\i (\mu(\ga)))}.
$$
\end{lemma}
\begin{proof} Without loss of generality, we may assume that $\nu_\psi(\La_{\mathsf u})>0$.
Recall that $\La_{\mathsf u}=\cup_{n\in\bb N}\La_{\mathsf u,n}$ 
where
$$
\La_{\u,n}=\bigcap_{m=1}^\infty\bigcup_{
\|\mu (\ga)\| \geq m, \ga\in\Ga_{{\mathsf u},n}} O_n(o,\ga o).
$$

Hence $\nu_\psi(\La_{\mathsf u,n})>0$ for some $n$. Now by Lemma \ref{lem.shadow}, we have for all $m\ge 1$,
$$
0<\nu_\psi(\La_{\mathsf u,n})\leq \sum_{
\substack{
\|\mu(\ga)\| \geq m,\\
\ga\in\Ga_{{\mathsf u},n}}}
\nu_\psi(O_n(o,\ga o))\ll \sum_{
\substack{
\|\mu(\ga)\|\geq m,\\
\ga\in\Ga_{{\mathsf u},n}}}
e^{-\psi(\mu(\ga))}.
$$
Since the implicit constant above is independent of $m$, it follows that the series $\sum_{\ga\in\Ga_{{\mathsf u},n}}e^{-\psi(\mu(\ga))}$ diverges, which implies the claim by \eqref{same}.
\end{proof}

The rest of this section is devoted to the  proof of the following:
\begin{prop}\label{bal}
Suppose that $\m$ is $\mathsf{u}$-balanced as defined in \eqref{bald}.
 If $\sum\limits_{\ga\in\Ga_{{\mathsf u},R}}e^{-\psi(\mu(\ga))}=\infty$ for some $R>0$, then $$\nu_{\psi}(\La_{\mathsf u})=1 =\nu_{\psi\circ \i} (\La_{\i({\mathsf u})}).$$
\end{prop}
Proof of this proposition involves investigating the relation between the $\mathsf{u}$-directional Poincare series
and the correlation function of $\m$ for the $a_t$-action.

\subsection*{Multiplicity of shadows}
\begin{lemma}\label{lem.ovl} 
 For any $R>0$ and $D>0$,  we have
$$
\sup_{T>0} \sum_{ \substack{\ga\in\Ga_{{\mathsf u},R},\\ T\leq \psi(\mu(\ga))\leq T+D} }\mathbbm 1_{O_{R}(o,\ga o)} <\infty.
$$
\end{lemma}

\begin{proof}
Suppose that there exist $\ga_1,\cdots,\ga_m\in\Ga_{{\mathsf u},R}$ and $k\in K$ such that $k^+\in \cap_{i=1}^m O_{R}(o,\ga_io)$
and $T\le  \psi(\mu(\ga_i))\le  T+D$.
By Lemma \ref{lem.f},
for all $1\leq i\leq m$, there exists $t_i\geq0$ such that  $ka_{t_i}o\in B(\ga_io,(2d+1)R)$.
Since $\ga_i^{-1}ka_{t_i}\in G_{(2d+1)R}$, we have $\norm{\mu(\ga_i)-t_iu}\leq d(2d+1)R$ by Lemma \ref{com}.
In particular, 
$$
t_i\psi({\mathsf u})\leq \psi(\mu(\ga_i))+\norm{\psi}d(2d+1)R\leq T+D+\norm{\psi}d(2d+1)R,
$$
and similarly 
$$t_i\psi({\mathsf u})\geq T-\norm{\psi}d(2d+1)R.$$
Hence $|\psi({\mathsf u})(t_i-t_1)|<2\norm{\psi}d(2d+1)R+D$.
Note that as $\psi({\mathsf u})>0$, for all $1\le i\le m$,
\begin{align*}
&d(\ga_io,\ga_1o)\leq d(\ga_io,ka_{t_i}o)+d(ka_{t_i}o,ka_{t_1}o)+d(ka_{t_1}o,\ga_1o)\\
&\leq 2d(2d+1)R+|t_i-t_1|\\ &
\leq S:=2d(2d+1)R+(\psi({\mathsf u}))^{-1}(2\norm{\psi}d(2d+1)R+D).
\end{align*}

Since there are only finitely many $\gamma_i o$ in a bounded ball of radius $S$,
it follows that $m$ is bounded above by a constant depending only on $S$. This proves the claim.
\end{proof}

\begin{cor}\label{lem.shell}
For any  large enough $R>0$, we have, for any $D>0$, 
$$
\sup_{T>0}\sum_{\substack{\ga\in\Ga_{{\mathsf u},R},\\ T\le \psi(\mu(\ga))\le T+D } }e^{-\psi(\mu(\ga))}
<\infty.
$$
\end{cor}
\begin{proof}
By Lemmas \ref{lem.shadow} and \ref{lem.f}, there exists $C=C(\psi)>0$ such that for all $R$ large enough, and any $T>0$,
\begin{align*}
\sum_{\substack{
\ga\in\Ga_{{\mathsf u},R},\\
 T\le \psi(\mu(\ga))\le T+D}}e^{-\psi(\mu(\ga))}&\leq 
 \sum_{\substack{\ga\in\Ga_{{\mathsf u},R},\\T\le \psi(\mu(\ga))\le T+D} }C\cdot 
 \nu_{\psi} (O_{R}(o,\ga o))<\infty
 \end{align*}
by Lemma \ref{lem.ovl}.
\end{proof}

\subsection*{Directional Poincare series}

For $r>0$ and $g\in G$, we set 
\begin{align}\label{eq.K0} 
Q_r&:=G_rA_r=KA_rKA_r,\text{ and }\\
\cal L_r(o,g(o)) &:=\{(h^+,h^-)\in \F^{(2)} : h\in G_{r} \cap  gG_{r}\exp (\br_- \u) \}.\notag
\end{align}
\begin{lem}\label{lem.qr}
For any $r>0$, we have $Q_r\subset G_{2r}$.
\end{lem}
\begin{proof} 
Let $g\in Q_r$ be arbitrary.
By definition, $g=k_1a_1k_2a_2$ for some $k_1,k_2\in K$ and $a_1,a_2\in A_r$.
Since
\begin{align*}
d(go,o)&= d(a_1k_2a_2o,o)\leq d(a_1k_2a_2o,a_1k_2o)+ d(a_1k_2o,o)\\
&= d(a_2o,o)+ d(o,a_1^{-1}o)<2r,
\end{align*}
the lemma follows.
\end{proof}
The following is the main ingredient of the proof of Proposition \ref{bal}: \begin{prop}\label{lem.EE}
Suppose that $\sum\limits_{\ga\in\Ga_{{\mathsf u},R}}e^{-\psi(\mu(\ga))}=\infty$ for some $R>0$.
If $r$ is large enough, we have the following for any $T>1$:
\be \label{eq2}
\int_0^T\int_0^T\sum_{\ga,\ga'\in\Ga} \tilde{\m}(Q_r\cap \ga Q_ra_{-t}\cap\ga' Q_r a_{-t-s})\,dt\,ds\ll\left(\sum_{\substack{\ga\in\Ga_{{\mathsf u},4d r}\\ \psi(\mu(\ga))\leq \delta T} }e^{-\psi(\mu(\ga))}\right)^{{\hspace{-0.15cm}2}}
;\ee
\be\label{eq1}
\int_0^T\sum_{\ga\in\Ga}\tilde{\m}(
Q_{4r}
\cap \ga 
Q_{4r}
a_{-t})\,dt\gg \sum_{ \substack{\ga\in\Ga_{{\mathsf u},r}\\
 \psi(\mu(\ga))< \delta T}}e^{-\psi(\mu(\ga))}
\ee
where the implied constants are independent of $T$.
\end{prop}

\begin{lemma}\label{lem.CC}
If $Q_r\cap\ga Q_ra_{-t}\neq\emptyset$ for $\ga\in\Ga$ and $t, r>0$, then 
$$
\norm{\mu(\ga)-t{\mathsf u}}\leq  4dr.
$$
\end{lemma}
\begin{proof}
If $Q_r\cap\ga Q_ra_{-t}\neq\emptyset$, there exist $q_0,q_0'\in 
Q_r
$ such that $q_0=\ga q_0'a_{-t}$.
By Lemma \ref{com} and Lemma \ref{lem.qr},
$$
t{\mathsf u}=\mu(a_t)=\mu(q_0^{-1}\ga q_0')\in\mu(\ga)+\fa_{4dr}.
$$
\end{proof}

In order to prove Proposition \ref{lem.EE}, we will  bound the integrals appearing in the lemma from below and above using shadows, and then apply the shadow lemma (Lemma \ref{lem.shadow}).
For this purpose, we observe several relations between the sets defined in \eqref{eq.K0} and shadows.

\begin{lemma}\label{lem.K0}
If $g\in Q_r\cap\ga Q_ra_{-t}$ for $\ga\in \Ga$ and $t, r > 0$, then
\begin{enumerate}
\item
$(g^+,g^-)\in\cal L_r(o,\ga o)$;
\item
$|\psi(\cal G(g^+,g^-))| <2 \norm{\psi}cr$ where $c$ is from  Lemma \ref{lem.BPS};
\item
$[g]A\cap Q_r\cap\ga Q_ra_{-t}\subset [g]
A_{2dr}
$.
\end{enumerate}
\end{lemma}
\begin{proof}
$(1)$ is immediate from the definition of $\cal L_r(o,\ga o)$.
Since $g\in Q_r$, $go\in B(o,2r)$ and hence
$\norm{\cal G(g^+,g^-)}<2cr$ by Lemma \ref{lem.BPS} and $(2)$ follows.
$(3)$ follows from the stronger inclusion $gA\cap Q_r\subset g
A_{2dr}
$ which follows from Lemma \ref{com} and Lemma \ref{lem.qr}.
\end{proof}
\begin{lemma}\label{lem.UI}
For any $g\in G$ and $r>0$, we have 
$$\cal L_r(o,g(o))\subset O_{4r}(o,g(o))\times O_{4r}(g(o),o).$$
\end{lemma}
\begin{proof} Let $(h^+,h^-)\in\cal L_r(o,g(o))$; so $h\in B(o, 2r)$ such that
$ha_to\in B(g(o),2r)$ for some $t\ge 0$.
Write $o=ha_0n_0o$ for some $a_0n_0\in AN$.
Since the Hausdorff distance between $a_0n_0A^+o$ and $A^+o$ is $d(a_0n_0o,o)$ \cite[1.6.6 (4)]{Eb}, we can find $q'\in 
ha_0n_0A^+o$ such that $d(q',h a_t o)<d(ha_0n_0o,ho)<2r$.
Hence, $d({g(o)},q')<d({g(o)},ha_to)+d(ha_to,q')<4r$ and it follows that $h^+\in O_{4r}(o,g(o))$.
A similar argument shows that $h^-\in O_{4r}(g(o),o)$.
\end{proof}

\begin{lemma}\label{lem.K1}
For all large enough $r>1$, we have for any $t>1$, 
$$\tilde{\m}(Q_r\cap\ga Q_ra_{-t})
\ll e^{-\psi(\mu(\ga))}$$
where the implied constant is independent of $t>1$.
\end{lemma}
\begin{proof}
If $r$ is large enough, we get by Lemma \ref{lem.UI}, Lemma \ref{lem.shadow} and Lemma \ref{lem.K0}:
\begin{align*}
&\tilde{\m}(Q_r\cap\ga Q_ra_{-t})\\
&=\int \mathbbm{1}_{Q_r\cap\ga Q_ra_{-t}}([{gb}])e^{\psi(\mathcal G(g^+,g^-))}d\nu_{\psi}(g^+)d\nu_{\psi\circ \i}(g^-) \,db\\
&=\int_{\cal L_r(o,\ga o)}\left(\int _A\mathbbm{1}_{Q_r\cap\ga Q_ra_{-t}}([gb])\,db\right)e^{\psi(\cal G(g^+,g^-))}d\nu_{\psi}(g^+)d\nu_{\psi\circ \i}(g^-)\\
&\leq \nu_\psi(O_{4r}(o,\ga o))\op{Vol}(
A_{2dr}
) e^{2\norm{\psi}cr}\\
&\ll  e^{-\psi(\mu(\ga))}.
\end{align*}
\end{proof}

\begin{lemma}\label{lem.K2}
If $Q_r\cap \ga Q_ra_{-t}\cap\ga' Q_ra_{-t-s}\neq\emptyset$ for $\ga,\ga'\in \Ga$ and $r, t, s>0$, then
\begin{enumerate}
\item
$\norm{\mu(\ga)-t{\mathsf u}},\,\norm{\mu(\ga^{-1}\ga')-s\u},\,\norm{\mu({\ga'})-(t+s)\u}\leq 4dr;$\\ 
\item
$\psi(\mu(\ga))+\psi(\mu(\ga^{-1}\ga'))\leq \psi(\mu(\ga'))+12dr\norm{\psi}$.
\end{enumerate}
\end{lemma}
\begin{proof}
Note that from the hypothesis, the intersections
$$
Q_r\cap \ga Q_ra_{-t},\text{ }Q_r\cap \ga^{-1}\ga' Q_ra_{-s},\text{ }Q_r\cap \ga' Q_ra_{-t-s}
$$
are all nonempty.
By Lemma \ref{lem.CC}, we obtain (1).

(2) follows since
\begin{align*}
    &|\psi(\mu(\ga))+\psi(\mu(\ga^{-1}\ga'))-\psi(\mu(\ga'))|\\
    &=|\psi(\mu(\ga)-t{\mathsf u})+\psi(\mu(\ga^{-1}\ga')-s\u)-\psi(\mu(\ga')-(t+s)\u)|\\
    &\leq 4dr\norm{\psi}+4dr\norm{\psi}+4dr\norm{\psi}=12dr\norm{\psi}.
\end{align*}
\end{proof}

 \subsection*{Proof of \eqref{eq2} in Proposition \ref{lem.EE}} Fix $s,t>0$.
 Let $r$ be large enough so that $\sum\limits_{\ga\in\Ga_{{\mathsf u},4dr }}e^{-\psi(\mu(\ga))}=\infty$.
In the following proof, the notation ${\overset{\prime \prime}{\sum}}$ means the sum
over all $(\ga,\ga')\in \G_{{\mathsf u}, 4dr}\times \G$ {satisfying: \begin{align*}
\ga^{-1}\ga'&\in\Ga_{{\mathsf u},4dr};\\
\psi(\mu(\ga))&\in (\delta t-4dr\|\psi\|,\delta t+4dr\|\psi\|); \; and\\
\psi(\mu(\ga^{-1}\ga'))&\in (\delta s-4dr\|\psi\|,\delta s+4dr\|\psi\|).\end{align*}}
Note that
\begin{align*}
&\sum_{\ga,\ga'\in\Ga}\tilde{\m}(Q_r\cap \ga Q_ra_{-t}\cap\ga' 
Q_r a_{-t-s})\\
=&\sum'' \tilde{\m}(Q_r\cap \ga Q_ra_{-t}\cap\ga' Q_ra_{-t-s})\, \text{by Lemma \ref{lem.K2}(1)} \\
\ll & \sum''   e^{-\psi(\mu(\ga'))}\qquad \text{by Lemma \ref{lem.K1}}\\
\ll &\sum'' e^{-\psi(\mu(\ga))}e^{-\psi(\mu(\ga^{-1}\ga'))}\qquad \text{by Lemma \ref{lem.K2}(2)} \\
\ll &  \left(\sum_{\substack{\ga\in\Ga_{{\mathsf u},4dr},\\\psi(\mu(\ga))\in (\delta t-c_0,\delta t+c_0)}
}e^{-\psi(\mu(\ga))}\right) \left(\sum_{\substack{\ga'\in\Ga_{{\mathsf u},4dr},\\ \psi(\mu(\ga'))\in (\delta s-c_0,\delta s+c_0)}}e^{-\psi(\mu(\ga'))}\right)
\end{align*}
 where $c_0=4dr\|\psi\|$.

Let $I_\ga$ denote the interval $\delta^{-1} [\psi(\mu(\ga))-c_0, \psi(\mu(\ga))+c_0]$. Note that $I_\ga\cap [0,T]\ne \emptyset$ implies that $\psi(\mu(\ga))\le \delta T+c_0$.
Hence
\begin{align*} &\int_0^T \left(\sum_{\substack{\ga\in\Ga_{{\mathsf u},4dr},\\\psi(\mu(\ga))\in (\delta t-c_0,\delta t+c_0)}} e^{-\psi(\mu(\ga))}\right) dt \\ 
&= \sum_{\ga\in\Ga_{{\mathsf u},4dr}} e^{-\psi(\mu(\ga))}  \int_0^T 
\mathbbm 1_{I_\ga} (t)  dt\ll
\sum_{\substack{\ga\in\Ga_{{\mathsf u},4dr},\\\psi(\mu(\ga))\le \delta T+c_0}} e^{-\psi(\mu(\ga))}.\end{align*}
Putting these two together {along} with Corollary \ref{lem.shell}, {used in order to remove $c_0$, concludes} the proof of \eqref{eq2}.
\qed

\begin{lem}\label{lem.LR}
For any $S>0$
and $r>0$, there exists $0<\ell(S,r)<\infty$ such that
for any $\ga\in \G$ with $\|\mu(\ga)\|>\ell(S,r)$, any point $(\xi,\eta)\in O_S(o,\ga o)\times O_S(\ga o,o)$ satisfies $\norm{\cal G(\xi,\eta)}<$
$\ell(S,r)$.
\end{lem}
\begin{proof}
Suppose not.
Then there exists a sequence $\ga_i\to\infty$ in 
$\Ga_{\mathsf{u},r}$ 
and $(\xi_i,\eta_i)\in O_S(o,\ga_io)\times O_S(\ga_i o,o)$ such that $\norm{\cal G(\xi_i,\eta_i)}\to\infty$.

We may write $\ga_i=k_ia_i\ell_i$ in $KA^+K$ decomposition, and assume that $k_i\to k_0$ after passing to a subsequence.
It follows that $\xi_i\to k_0^+$ and $\eta_i\to\eta_0$ for some $\eta_0\in \cal F$ such that $(k_0^+,\eta_0)\in\cal F^{(2)}$
as $\ga_i\to\infty$ regularly, by Lemma \ref{lem.F2}.
Hence $\lim_{i\to \infty} \norm{\cal G(\xi_i,\eta_i)}= \norm{\cal G(k_0^+,\eta_0)}<\infty$, which is a contradiction.
\end{proof}

In the following, we fix a large number $S_0$ which satisfies Lemma \ref{lem.shadow}. For each $r>1$, let $\ell_r:=\ell(S_0, r)>0$ be as provided by Lemma \ref{lem.LR} so that for any $(\xi,\eta)\in \bigcup_{ \ga\in \Ga_{\mathsf{u}, r}, \|\mu(\ga)\|\ge \ell_r} O_{S_0}(o,\ga o)\times O_{S_0}(\ga o,o)$, we have $\norm{\cal G(\xi,\eta)}<\ell_r$.

\begin{lem}\label{lem.LB}
If $r>1$ is large enough, the following holds:
 for any
$(\xi,\eta)\in  O_{S_0}(o,\ga o)\times  O_{S_0}(\ga o,o)$
for some $\ga\in \Ga_{{\mathsf u},r}$ with $ \|\mu(\ga)\| \ge \ell_r$,
there exist $t\in\bb R$ and $g\in Q_{2r}$ such that $$ga_{[t-1,t+1]}\subset\ga Q_{2r}\;\;\text{  and  }\;\; (g^+,g^-)=(\xi,\eta).$$
\end{lem}
\begin{proof}
Let $(\xi,\eta)$ be as in the statement. Then by Lemma \ref{lem.f},
there exists $t\geq 0$ and $k\in K$ such that $\xi=k^+$, 
$ka_{t}o\in B(\ga o,r+(d+1)S_0)$.
Let $g\in G$ be such that $(g^+,g^-)=(\xi,\eta)$.
Since $\|\mu(\ga)\|>\ell_r$, by replacing $g\in G$ by an element of $gA$, we may assume that  $d(go,o)<c\ell_r+c'$  where $c$ and $c'$ are as in  Lemma \ref{lem.BPS}. As $g^+=k^+$ and hence $k^{-1}g\in P$, it follows by \cite[1.6.6 (4)]{Eb} that
$d(ga_{t}o,ka_{t}o)\leq d(go,o)$ for all $t\geq0$.

Hence for all $s\in [t-1,t+1]$,
$$
d(ga_{s}o,ka_{t}o)<d(ga_so,ga_to)+d(ga_{t}o,ka_{t}o)\leq 1+d(go,o)<1+c\ell_r+c'.
$$
It follows that
$ga_{[t-1,t+1]}\in  \ga G_{r+(d+1)S_0 +c\ell_r +c'+1}$.
Now if $r$ is large enough, 
$$ga_{[t-1,t+1]}\subset \ga Q_{2r}.$$
Similarly, since $go\in 
G_{c\ell_r+c'}$, we have
$g\in Q_{2r}$,
which was to be shown.
\end{proof}

\begin{lemma}\label{lem.K3}
If $r>1$ is large enough, the following holds:
 for any $g\in G$ such that
$(g^+,g^-)\in O_{S_0}(o,\ga o)\times  O_{S_0}(\ga o,o)$
 for some $\ga\in\Ga_{{\mathsf u},r}$ and $T>0$ satisfying
$$ \|\mu(\ga)\|>\ell_r\;\;\text{ and }\;\;
8dr\norm{\psi}+\delta<\psi(\mu(\ga))<\delta T-8dr\norm{\psi}-\delta,
$$
we have \be\label{ggg}
\int_0^T \int_A \mathbbm{1}_{
Q_{4r}
\cap\ga 
Q_{4r}
a_{-t}}([gb])\,db\,dt\geq  2\op{Vol}(
A_{2r}
).
\ee
\end{lemma}
\begin{proof}
Note that replacing $g$ with an element of $gA$ does not affect the validity of \eqref{ggg}.
Hence by Lemma \ref{lem.LB}, we may assume that $g\in Q_{2r}$ and $ga_{[t_0-1,t_0+1]}\subset\ga Q_{2r}$ for some $t_0\in\mathbb R$.

It follows that
$Q_{2r}\cap\ga Q_{2r}a_{-t}\neq\emptyset$  for all $t\in [t_0-1,t_0+1]$.
Note that $|\psi(\mu(\ga))-t_0\delta|\leq 8dr\norm{\psi}$ by Lemma \ref{lem.CC},
and hence $[t_0-1,t_0+1]\subset [0,T]$ by the hypothesis.
Since $g\in Q_{2r}$ and hence $g\in G_{4r}$ by Lemma \ref{lem.qr}, we have $gA\cap Q_{4r}\supset gA_{4r}$.
Consequently,
\begin{equation}\label{eq.ovl}
\int_A \mathbbm{1}_{
Q_{4r}
\cap\ga 
Q_{4r}
a_{-t}}([gb])\,db\geq \int_{
A_{4r}
}\mathbbm{1}_{\ga 
Q_{4r}
}([gba_t])\,db.
\end{equation}

By definition of $Q_{4r}$, there is a uniform lower bound  for \eqref{eq.ovl}, say 
$\op{Vol}(A_{2r})$,
whenever $[ga_t]\cap\ga 
Q_{4r}
\neq\emptyset$, in particular for all $t\in[t_0-1,t_0+1]$ by Lemma \ref{lem.LB}.
Hence,
\begin{align*}
&\int_0^T \int_A \mathbbm{1}_{
Q_{4r}
\cap\ga 
Q_{4r}
a_{-t}}([gb])\,db\,dt\\
&\geq \int_{t_0-1}^{t_0+1} \int_A \mathbbm{1}_{
Q_{4r}
\cap\ga 
Q_{4r}
a_{-t}}([gb])\,db\,dt\geq 2
\op{Vol}(A_{2r}).
\end{align*}
This proves the lemma.
\end{proof}

 \subsection*{Proof of \eqref{eq1} in Proposition \ref{lem.EE}}
By definition of ${\tilde{\m}}$, we have for any $\ga\in \Ga$ and $r, t>0$,
\begin{align*}
&\tilde{\m}(
Q_{4r}
\cap\ga 
Q_{4r}
a_{-t})\\
= &\int_{{\F^{(2)}}}\left(\int_A \mathbbm{1}_{
Q_{4r}
\cap\ga 
Q_{4r}
a_{-t}}([gb])\,db\right)e^{\psi(\cal G(g^+,g^-))}d\nu_{\psi}(g^+)d\nu_{\psi\circ \i}(g^-)\\
\geq& \int_{
O_{S_0}(o,\ga o)\times  O_{S_0}(\ga o,o)
}\left(\int_A \mathbbm{1}_{
Q_{4r}
\cap\ga 
Q_{4r}
a_{-t}}([gb])\,db\right)e^{\psi(\cal G(g^+,g^-))}d\nu_{\psi}(g^+)d\nu_{\psi\circ \i}(g^-).
\end{align*}

Now Lemma \ref{lem.K3} implies that if $\ga\in\Ga_{{\mathsf u},r}$, 
$\|\mu(\ga)\|>\ell_r$ and $(8dr\norm{\psi}+\delta)<\psi(\mu(\ga))<\delta T-(8dr\norm{\psi}+\delta)$, then
\begin{align*}
&\int_0^T \tilde{\m}(
Q_{4r}
\cap\ga 
Q_{4r}
a_{-t})\,dt\\
&\geq
2\op{Vol}(
A_{2r}
)\int_{
O_{S_0}(o,\ga o)\times  O_{S_0}(\ga o,o)
}e^{\psi(\cal G(g^+,g^-))}d\nu_{\psi}(g^+)d\nu_{\psi\circ \i}(g^-)\\
&\geq 2\op{Vol}(
A_{2r}
)e^{-\norm{\psi}\ell_r} \nu_\psi(O_{S_0}(o,\ga o))\nu_{\psi\circ \i}(O_{S_0}(\ga o,o))\\
&\geq 2\op{Vol}(
A_{2r}
)e^{-\norm{\psi}\ell_r} \beta(\nu_{\psi\circ \i})c_1 e^{-\kappa\norm{\psi}S_0}e^{-\psi(\mu(\ga))},
\end{align*}
where the second inequality follows from 
the lower bound 
$e^{\psi(\cal G(g^+,g^-))}\geq e^{-\norm{\psi}\ell_r}$ and the last inequality follows from Lemma \ref{lem.shadow}. 
Therefore,
\begin{align*}
\int_0^T\sum_{\ga\in\Ga}\tilde{\m}(
Q_{4r}
\cap \ga 
Q_{4r}
a_{-t})\,dt &\geq\int_0^T\sum_{\ga\in\Ga_{{\mathsf u},r},\,\|\mu(\ga)\|>\ell_r}\tilde{\m}(
Q_{4r}
\cap \ga 
Q_{4r}
a_{-t})\,dt\\
&\gg  \sum_{
\substack{
\ga\in\Ga_{{\mathsf u},r}\,\|\mu(\ga)\|>\ell_r,\\
\psi(\mu(\ga))<\delta T-(8dr\norm{\psi}+\delta)
}
}  e^{-\psi(\mu(\ga))}.
\end{align*}
Since $\#\{\ga\in \Ga:\|\mu(\ga)\|\le \ell_r\}$ is a finite set,
this proves the lemma by Corollary \ref{lem.shell}.
\qed

Proposition \ref{lem.EE} yields:
\begin{cor}\label{cobal}  Suppose that for any large $r, s\gg 1$, and $T>1$,
\be \label{eq_assump}
 \int_0^T \sum_{\ga\in \Ga} \tilde{\mathsf m}(Q_r\cap \ga Q_r a_{-t})\,dt \asymp  \int_0^T \sum_{\ga\in \Ga} \tilde{\mathsf m} (Q_s \cap \ga Q_s a_{-t})\,dt
\ee
with the implied constant independent of $T$.
If  $\sum\limits_{\ga\in\Ga_{{\mathsf u},R}}e^{-\psi(\mu(\ga))}=\infty$ for some $R>0$, then for all sufficiently large $r$, we have for any $T>1$:
\begin{multline} \label{eq3}
\int_0^T\int_0^T\sum_{\ga,\ga'\in\Ga} \tilde{\m}(Q_r\cap \ga Q_ra_{-t}\cap\ga' Q_ra_{-t-s})\,dt\,ds\ll \\ \left(\int_0^T\sum_{\ga\in\Ga}\tilde{\m}(Q_r\cap \ga Q_ra_{-t})\,dt \right)^2.
\end{multline}
\end{cor}
\begin{proof}
By Proposition \ref{lem.EE}, we get
\begin{multline*} 
\int_0^T\int_0^T\sum_{\ga,\ga'\in\Ga} \tilde{\m}(Q_r\cap \ga Q_ra_{-t}\cap\ga' Q_ra_{-t-s})\,dt\,ds\ll \\ \left(\int_0^T\sum_{\ga\in\Ga}\tilde{\m}(Q_{16dr}\cap \ga Q_{16dr}a_{-t})\,dt \right)^2,
\end{multline*}
which implies the claim in view of the hypothesis \ref{eq_assump}.
\end{proof}

\subsection*{Proof of Proposition \ref{bal}}
We will apply the following version of {the} Borel-Cantelli lemma:
\begin{lemma}\cite[Lem. 2]{AS}\label{lem.BC}
Let $(\Omega,\mathsf M)$ be a finite Borel measure space and $\{P_t:t\geq0\}\subset\Om$ be such that $(t,\omega)\mapsto\mathbbm{1}_{P_t}(\omega)$ is measurable.
Suppose that
\begin{enumerate}
\item
$\int_0^\infty \mathsf M(P_t)\,dt=\infty$, and
\item
there is a constant $C>0$ such that
$$
\int_0^T\int_0^T \M(P_t\cap P_s)\,dt\,ds\leq C\left(\int_0^T \M(P_t)\,dt\right)^2\text{ for all }T\gg1.
$$
\end{enumerate}
Then we have
$$
\M\left\{\omega\in\Om: \int_0^\infty \mathbbm{1}_{P_t}(\omega)\,dt=\infty\right\}>\frac{1}{C}.
$$
\end{lemma}

Suppose that  $\sum_{\ga\in \Ga_{{\mathsf u},R} }e^{-\psi(\mu(\ga))}=\infty$ for some $R>0$. 
Let $r>R$ be large enough to satisfy Proposition \ref{lem.EE}, and consider $Q_r=G_rA_r$. As $M$ commutes with $A$ and $Q_r=KA_r^+KA_r$, $Q_r$ is an $M$-invariant subset. Let $[Q_r]=\Gamma\ba Q_r/M \subset \Gamma\ba G/M$.
Set
$$
\M:=\mathsf m|_{[Q_r]}\text{ and }P_t:=\Ga \ba \Ga (Q_r\cap \Ga Q_r a_{-t})\subset \Ga\ba G/M.
$$
We claim that 
\be\label{re}
\int_0^T\int_0^T \M(P_t\cap P_s)\,ds\,dt\ll\left(\int_0^T \M(P_t)\,dt\right)^2.
\ee
Since $\m$ is assumed to be $\mathsf{u}$-balanced, Corollary \ref{cobal} applies, and hence
\begin{equation}
\int_0^T\int_0^T \M(P_t\cap P_{t+s})\,ds\,dt\ll\left(\int_0^T \M(P_t)\,dt\right)^2.
\end{equation}
Therefore
\begin{align*}
&\int_0^T\int_0^T \M(P_t\cap P_s)\,ds\,dt=2\int_0^T\int_t^T \M(P_t\cap P_s)\,ds\,dt\\
&\leq 2\int_0^T\int_0^T \M(P_t\cap P_{t+s})\,ds\,dt\ll\left(\int_0^T \M(P_t)\,dt\right)^2,
\end{align*}
proving the claim.
Applying Lemma \ref{lem.BC} with $\M$ and $P_t$, we conclude that
$$
\mathsf m\left\{[g]\in [Q_r]: \int_0^\infty \mathbbm{1}_{[Q_r]}([g]a_t)dt =\infty\right\}>0.
$$
It follows that  $\nu_\psi(\{g^+\in \F : \limsup [g] a_t \ne \emptyset\} )>0$ and
hence $\nu_\psi (\La_{\mathsf u})>0$. On the other hand, by
\eqref{same}, we have
 $\sum_{\ga\in \Ga_{\i({\mathsf u}),R} }e^{-\psi\circ \i (\mu(\ga))}=\infty$.
By the same argument as above, this implies that
$$\nu_{\psi\circ \i} (\{g^+\in \F : \limsup [g] \exp (t \i ({\mathsf u})) \ne \emptyset\} )>0$$
and hence
$\nu_{\psi \circ \i}(\La_{\i ({\mathsf u})})>0$.
This finishes the proof by Proposition \ref{lem.01}. 

\subsection*{Proof of Theorem \ref{dio}}.
The equivalence $(1) \Leftrightarrow (2)$ follows from Proposition
\ref{lem.01}.
The equivalence $(2) \Leftrightarrow (3)$
follows from Proposition \ref{lem.dich}.
The equivalence $(3) \Leftrightarrow (4)\Leftrightarrow (5)$ follows from
Proposition \ref{mini}.
The implication $(1)\Rightarrow (6)$ is proved in Lemma \ref{lem.se},
and the implication $(6)\Rightarrow (7)$ follows from Lemma \ref{lem.se} and Proposition \ref{bal}.

\begin{Rmk}\label{518}
The asymptotic inequality \eqref{re} shows that
if $\m$ is $\mathsf{u}$-balanced and $\sum\limits_{\ga\in\Ga_{{\mathsf u},R}}e^{-\psi(\mu(\ga))}=\infty$ for some $R>0$, then the measure preserving flow $(\Ga\ba G/M,\m, \{a_t\})$ is rationally ergodic and the following
$$\cal A_T=\frac{1}{\m ([ Q_r])^2} \int_{\Gamma\ba G/M} \int_0^T \mathbbm 1_{[Q_r]}(xa_t )dt d\m (x)$$
is the asymptotic type of the flow in the sense of \cite{Aa} and \cite[5]{AS}.
\end{Rmk}
\section{Dichotomy for Anosov groups}
Let $\Ga<G$ be an Anosov subgroup defined as in the introduction. 
%We then have $\L_\Ga-\{0\}\subset \inte\fa^+$, which implies that $\Ga$ is regular \cite[Prop. 4.6]{PS}.
 For each $\mathsf{v}\in \inte \L_\Ga$,
there exists a unique $\psi_{\mathsf v}\in \fa^*$ such that
$\psi_{\mathsf v}\ge \psi_\Ga$ and $\psi_{\mathsf v} ({\mathsf v})=\psi_\Ga({\mathsf v})$, and a unique
$(\Ga, \psi_{\mathsf v})$-conformal measure $\nu_{\mathsf v}$ supported on $\La$ (\cite{Samb2}, \cite{ELO}).
Moreover, $\{u\in \inte \fa^+:\psi_{\mathsf v}({\mathsf u})=\psi_\Ga({\mathsf u})\}=\br_+ {\mathsf v}$ (\cite{Q4}, \cite{Samb3}).
The assignments $\mathsf{v}\mapsto \psi_{\mathsf v} $ and $\mathsf{v}\mapsto \nu_{\mathsf v}$ give bijections among
$\inte \L_\Ga$, $D_\Ga^\star$ and
the space of all $\Gamma$-conformal measures supported on $\La$ \cite[Prop. 4.4 and Thm. 7.7]{LO}.

For each $\mathsf{v}\in \inte\L_\Ga$, we denote by $\m_{\mathsf v}$  the BMS measure on $\Ga\ba G/M$ associated to $(\nu_{\mathsf v}, \nu_{\i({\mathsf v})})$.
Chow and Sarkar proved the following theorem for $f_1, f_2\in C_c(\Gamma\ba G/M)$.
\begin{thm}\cite{CS} \label{lem.mix}
Let $\Ga<G$ be an Anosov subgroup and let $\mathsf{v}\in \inte \L_\Ga$.
There exists $\kappa_{\mathsf v}>0$ such that for any $f_1,f_2\in C_c(\Ga\ba G/M)$,
$$
\lim_{t\to+\infty}t^{\frac{\op{rank}(G)-1}{2}}\int_{\Ga\ba G/M}f_1(x)f_2(x\exp(t{\mathsf v}))\,d\m_{\mathsf v}(x)=\kappa_{\mathsf v}\cdot \mathsf m_{\mathsf v}(f_1)\mathsf m_{\mathsf v}(f_2).
$$
\end{thm}
Since $\m_{\mathsf v}$ is $A$-invariant, the above is equivalent to:
\be\label{minus}
\lim_{t\to+\infty}t^{\frac{\op{rank}(G)-1}{2}}\int_{\Ga\ba G/M}f_1(x)f_2(x\exp(-t\v))\,d\m_{\mathsf v}(x)=\kappa_{\mathsf v}\cdot \mathsf m_{\mathsf v}(f_1)\mathsf m_{\mathsf v}(f_2).
\ee

 In particular, for any $\mathsf{v}\in \inte\L_\Ga$, the measure $\m_{\mathsf v}$ is $\mathsf{v}$-balanced.

\begin{cor}\label{lem.DC}
For any $\mathsf{v}\in \inte\L_\Ga$ and any bounded Borel subset $Q \subset G/M$ with $\tilde{\mathsf m}_{\mathsf v}(\op{int}Q)>0$,  we have 
$$\int_0^\infty  \sum_{\ga\in \Ga}\tilde{\mathsf m}_{\mathsf v}(Q\cap\ga Q\exp(-t\v))\,dt=\infty\text{
if and only if } \op{rank}(G)\leq 3.$$
\end{cor}
\begin{proof} Choose $\tilde f_1, \tilde f_2\in C_c(G/M)$ so that $0\le \tilde f_1\le \mathbbm 1_{Q}\le \tilde f_2$ and $\tilde\m_{\mathsf v}(
\tilde{f_1}
)>0$. For each $i=1,2$, let $f_i\in C_c(\Ga\ba G/M)$ {be} defined by
$f_i([g])=\sum_{\ga\in \Ga} \tilde f_i(\ga g)$. 
By \eqref{minus}, we get 
\begin{align*} &\int_{\Ga\ba G/M}f_i
([g]
\exp(t{\mathsf v})
)f_i([g]) d{\m}_{\mathsf v} [g]\\ &=\int_{G/M}\sum_{\ga\in \Ga} \tilde f_i
(g
\exp(t{\mathsf v})
)\tilde f_i(g) d\tilde{\m}_{\mathsf v} (g) \asymp t^{(-\op{rank}(G) +1)/2} .\end{align*} The claim follows since $\int_{{1}}^\infty t^{(-\op{rank}(G)+1)/2} dt=\infty$ if and only if $\op{rank}(G)\!\le~\!\!3$.
\end{proof}

By Theorem \ref{dio}, the following theorem implies Theorem \ref{thm.Ano}:
\begin{thm}\label{cor.D}
Let $\mathsf{v}\in \inte\L_\Ga$ and $\mathsf{u}\in \inte\fa^+$.
The following are equivalent:
\begin{enumerate}
    \item $\op{rank}(G)\leq 3$ and $\br {\mathsf u}=\br {\mathsf v}$;
\item $\sum_{\ga\in\Ga_{{\mathsf u},R}}e^{-\psi_{\mathsf v}(\mu(\ga))}=\infty$ for some $R>0$.
\end{enumerate}
\end{thm}
\begin{proof} Suppose that $\op{rank}(G)\leq 3$ and $\mathsf{u}={\mathsf{v}}$. Let $a_t=\exp (t{\mathsf v})$.
Let $Q_r\subset G/M$ be as in \eqref{eq2} of Proposition \ref{lem.EE}. Then for $\delta=\psi_{\mathsf v}({\mathsf v})>0$, we have
\be\label{note}\int_0^T\int_0^T\sum_{\ga,\ga'\in\Ga} \tilde{\m}_{\mathsf v}(Q_r\cap \ga Q_ra_{-t}\cap\ga'  Q_ra_{-t-s})\,dt\,ds\ll\left(\sum_{\substack{\ga\in\Ga_{{\mathsf v},4dr}\\ \psi(\mu(\ga))\leq \delta T} }e^{-\psi_{\mathsf v}(\mu(\ga))}\right)^{{\hspace{-0.15cm}2}}.
\ee

Set  $Q_r^-:=\cap_{0\leq s\leq r/10}Q_ra_{-s}$.
We may assume that $\mathsf m_{\mathsf v}(\op{int}Q_r^-)>0$ by increasing $r$.
Note that
\begin{multline*}
\frac{r}{10} \int_0^T  \sum_{\ga \in\Ga}\tilde{\m}_{\mathsf v}(Q_r^-\cap \ga Q_r^-a_{-t})\,dt \leq \\
\int_0^T \int_{0\leq s \leq r/10} \sum_{\ga \in\Ga}
\tilde{\m}_{\mathsf v}(Q_r\cap \ga (Q_r \cap Q_r a_{-s})a_{-t})\,ds \,dt .
\end{multline*}
By \eqref{note}, we get
$$
\int_0^T  \sum_{\ga \in\Ga}\tilde{\mathsf m}_{\mathsf v}(Q_r^-\cap \ga Q_r^-a_{-t})\,dt \ll \left( \sum_{\substack{\ga\in\Ga_{{\mathsf v},4dr },\\
 \psi_{\mathsf v}(\mu(\ga))<\delta T}} e^{-\psi_{\mathsf v}(\mu(\ga))}\right)^{{\hspace{-0.15cm}2}}.
$$

Hence by Corollary \ref{lem.DC}, we get $\sum_{\ga\in\Ga_{{\mathsf v},R}}e^{-\psi_{\mathsf v}(\mu(\ga))}=\infty$.

Now suppose that $\sum_{\ga\in\Ga_{{\mathsf u},R}}e^{-\psi_{\mathsf v}(\mu(\ga))}=\infty$ for some $R>0$.
By Lemma \ref{lem.neq}, $\psi_{\mathsf v}({\mathsf u})=\psi_\Ga ({\mathsf u})$.
This implies $\br {\mathsf v}=\br {\mathsf u}$, as $\br {\mathsf v}$ is the unique line where $\psi_{\mathsf v}$ and $\psi_\Ga$ are equal to each other. This also implies $\mathsf{u}\in \inte\L_\Ga$. By Proposition \ref{lem.EE}, it follows that 
$ \int_0^\infty \sum_{\ga\in \Ga} \tilde {\m}_{\mathsf v}(Q_r\cap \ga Q_ra_{-t})\,dt=\infty$. Hence $\text{rank} (G)\le 3$ by Corollary \ref{lem.DC}.
\end{proof}

\begin{Rmk} It follows from Theorem \ref{cor.D} that when $\op{rank} G\le 3$ and $\mathsf{v}\in \inte\L_\Ga$,
the flow $(\Gamma\ba G/M, \m_{\mathsf v}, \exp (t{\mathsf v}))$ is rationally ergodic by Remark \ref{518}.
\end{Rmk}

\end{document}